\RequirePackage[clockwise]{rotating}
\documentclass[a4paper, reqno]{amsart}
\DeclareRobustCommand{\SkipTocEntry}[5]{}
\usepackage{amsmath,amsthm,amssymb,relsize}
\usepackage[nobysame,alphabetic,initials]{amsrefs}
\usepackage
[left=30mm,
right=30mm,
top=30mm,
bottom=32mm]
{geometry}
\usepackage{mathrsfs,mathtools,makecell}

\usepackage{booktabs, multirow, adjustbox, array,xpatch, comment}

\newcolumntype{R}[2]{%
    >{\adjustbox{angle=#1,lap=\width-(#2)}\bgroup}%
    l%
    <{\egroup}%
}
\usepackage{pifont}
\usepackage{floatrow}
\DeclareFloatFont{tiny}{\small}% "scriptsize" is defined by floatrow, "tiny" not
\usepackage{enumerate}

\usepackage{tikz}
\usetikzlibrary{
  cd,
  calc,
  positioning,
  fit,
  arrows,
  decorations.pathreplacing,
  decorations.markings,
  shapes.geometric,
  backgrounds,
  bending
}
\usepackage{tikzsymbols}
\usepackage{xcolor}

\definecolor{darkblue}{rgb}{0,0,0.6}

\usepackage[breaklinks, pdftex, ocgcolorlinks,colorlinks=true, citecolor=darkblue, filecolor=darkblue, linkcolor=darkblue, urlcolor=darkblue]{hyperref}
\usepackage{footnotebackref}
\makeatletter
\xpretocmd{\@adminfootnotes}{\let\@makefntext\BHFN@OldMakefntext}{}{}
\renewcommand\@makefntext[1]{%
  \@ifundefined{@makefnmark}
    {}
    {%
     \renewcommand\@makefnmark{%
       \mbox{%
         \textsuperscript{%
           \normalfont
           \hyperref[\BackrefFootnoteTag]{\@thefnmark}%
         }%
       }\,%
     }%
     \BHFN@OldMakefntext{#1}%
  }%
}
\makeatother

\usepackage[capitalize,noabbrev]{cleveref}

\makeatletter
\newtheorem*{rep@theorem}{\rep@title}
\newcommand{\newreptheorem}[2]{%
\newenvironment{rep#1}[1]{%
 \def\rep@title{#2 \ref{##1}}%
 \begin{rep@theorem}}%
 {\end{rep@theorem}}}
\makeatother

\newtheorem{proposition}{Proposition}[section]
\newtheorem{theorem}[proposition]{Theorem}
\newtheorem{corollary}[proposition]{Corollary}
\newtheorem{lemma}[proposition]{Lemma}

\theoremstyle{definition}
\newtheorem{definition}[proposition]{Definition}
\newtheorem{question}[proposition]{Question}
\newtheorem{example}[proposition]{Example}

\theoremstyle{remark}
\newtheorem{remark}[proposition]{Remark}

\newtheorem*{remark*}{Remark}

\newreptheorem{theorem}{Theorem}
\newreptheorem{lemma}{Lemma}
\newreptheorem{proposition}{Proposition}
\newreptheorem{corollary}{Corollary}
\newreptheorem{question}{Question}

\AddToHook{env/proposition/begin}{\crefalias{proposition}{proposition}}
\AddToHook{env/theorem/begin}{\crefalias{proposition}{theorem}}
\AddToHook{env/thmx/begin}{\crefalias{proposition}{thmx}}
\AddToHook{env/corx/begin}{\crefalias{proposition}{corx}}
\AddToHook{env/assumption/begin}{\crefalias{proposition}{assumption}}
\AddToHook{env/hypothesis/begin}{\crefalias{proposition}{hypothesis}}
\AddToHook{env/corollary/begin}{\crefalias{proposition}{corollary}}
\AddToHook{env/definition/begin}{\crefalias{proposition}{definition}}
\AddToHook{env/lemma/begin}{\crefalias{proposition}{lemma}}
\AddToHook{env/question/begin}{\crefalias{proposition}{question}}
\AddToHook{env/example/begin}{\crefalias{proposition}{example}}
\AddToHook{env/conjecture/begin}{\crefalias{proposition}{conjecture}}
\AddToHook{env/remark/begin}{\crefalias{proposition}{remark}}
\AddToHook{env/const/begin}{\crefalias{proposition}{const}}

\crefname{assumption}{Assumption}{Assumptions}
\crefname{hypothesis}{Hypothesis}{Hypotheses}
\crefname{theorem}{Theorem}{Theorems}
\crefname{thmx}{Theorem}{Theorems}
\crefname{proposition}{Proposition}{Propositions}
\crefname{corollary}{Corollary}{Corollaries}
\crefname{definition}{Definition}{Definitions}
\crefname{lemma}{Lemma}{Lemmas}
\crefname{question}{Question}{Questions}
\crefname{example}{Example}{Examples}
\crefname{conjecture}{Conjecture}{Conjectures}
\crefname{remark}{Remark}{Remarks}
\crefname{const}{Construction}{Constructions}

\newcommand{\Diff}{\text{Diff}}

\renewcommand{\top}{\mathrm{Top}}

\newcommand{\N}{\mathbb{N}}

\newcommand{\R}{\mathbb{R}}
\newcommand{\Z}{\mathbb{Z}}

\newcommand{\im}{\operatorname{Im}}
\newcommand{\Id}{\operatorname{Id}}
\newcommand{\id}{\operatorname{Id}}

\newcommand{\spin}{\mathrm{Spin}}
\newcommand{\pinplus}{{\mathrm{Pin}^+}\!}
\newcommand{\pinminus}{{\mathrm{Pin}^-}}

\newcommand{\Spin}{\mathrm{Spin}}

\newcommand{\onto}{\twoheadrightarrow}
\newcommand{\into}{\hookrightarrow}
\newcommand{\ol}{\overline}
\newcommand{\wt}{\widetilde}

\newcommand{\sm}{\setminus}

\newcommand{\RP}{\mathbb{RP}}

\DeclareMathOperator{\Sq}{Sq}

\DeclareMathOperator{\Hom}{Hom}

\DeclareMathOperator{\BSpin}{BSpin}
\DeclareMathOperator{\TopSpin}{TopSpin}

\DeclareMathOperator{\ev}{ev}
\DeclareMathOperator{\BTop}{BTop}
\newcommand{\Top}{\mathrm{Top}}

\newcommand{\OO}{\mathrm{O}}

\DeclareMathOperator{\BO}{BO}

\DeclareMathOperator{\pr}{pr}

\newcommand{\bsm}{\left(\begin{smallmatrix}}
\newcommand{\esm}{\end{smallmatrix}\right)}

\usepackage{letltxmacro}

\LetLtxMacro\Oldfootnote\footnote

\begin{document}
\title{Stably exotic fillings of 3-manifolds}

\author{Daniel Kasprowski}
\address{School of Mathematical Sciences, University of Southampton, United Kingdom}
\email{d.kasprowski@soton.ac.uk}

\author{Patrick Orson}
\address{Department of Mathematics, California Polytechnic State University, San Luis Obispo, USA}
\email{porson@calpoly.edu}

\author{Mark Powell}
\address{School of  Mathematics and Statistics, University of Glasgow, United Kingdom}
\email{mark.powell@glasgow.ac.uk}

\author{Arunima Ray}
\address{School of Mathematics and Statistics, The University of Melbourne, Australia}
\email{aru.ray@unimelb.edu.au}

\def\subjclassname{\textup{2020} Mathematics Subject Classification}
\expandafter\let\csname subjclassname@1991\endcsname=\subjclassname

\begin{abstract}
We investigate which 3-manifolds bound 4-manifolds that are homeomorphic but not stably diffeomorphic, where stabilising means taking connected sum with copies of $S^2\times S^2$. We show that every closed, orientable 3-manifold admits such fillings, as do certain families of nonorientable 3-manifolds. In contrast we show that for a 3-manifold containing a 2-sided~$\RP^2$, any two smooth, homeomorphic fillings are stably diffeomorphic. 
\end{abstract}

\maketitle

\section{Introduction}

A \emph{filling} of a closed, smooth 3-manifold $Y$ is a smooth, compact 4-manifold $M$ with $\partial M = Y$. We prove the following. 

\begin{theorem}\label{thm:some-stably-exotic-fillings-exist}
Let $Y$ be a closed, smooth 3-manifold. Then $Y$ admits a pair of fillings that are homeomorphic rel.\ boundary but not stably diffeomorphic rel.\ boundary if one of the following holds.
  \begin{enumerate}[(i)]
    \item\label{main-thm-item-i} $Y$ is orientable.  
    \item\label{main-thm-item-ii} $Y$ is a circle bundle over a closed %nonorientable
    surface $F \neq \RP^2$. 
    \item\label{main-thm-item-iii} $Y$ admits a null-bordant tangential $\pinplus$ structure. 
   \end{enumerate}
 \end{theorem}

Here two fillings $M$ and $M'$ of $Y$ are \emph{stably diffeomorphic rel.\ boundary} if there exists $k\geq 0$ such that $M \# k(S^2 \times S^2)$ and $M' \# k(S^2 \times S^2)$ are diffeomorphic via a diffeomorphism restricting to the identity on the boundary.
Stable homeomorphism, which is used later, is defined analogously.  
Note that the fillings in \cref{thm:some-stably-exotic-fillings-exist} are in particular exotic fillings. By a theorem of Gompf \cite{Gompf84}, homeomorphism implies stable diffeomorphism for compact, orientable 4-manifolds, so the fillings we produce in~\cref{thm:some-stably-exotic-fillings-exist} are necessarily nonorientable.

Every orientable 3-manifold admits a null-bordant tangential $\pinplus$ structure, so~\cref{thm:some-stably-exotic-fillings-exist}~\eqref{main-thm-item-i} is implied by \cref{thm:some-stably-exotic-fillings-exist}~\eqref{main-thm-item-iii}. We state \eqref{main-thm-item-i} separately for emphasis, and because we give the proof for this case separately, in \cref{section:orientable-case}. 

Our examples will be of the form
\[M = Z \# K_3 \text{ and } M' = Z \# 11(S^2 \times S^2), \]
for a suitable nonorientable filling $Z$ of $Y$, where $K_3$ denotes the Kummer surface. The first examples of this sort, for closed 4-manifolds, are due to Kreck~\cite{Kreck84}. 
Using that $K_3$ splits topologically as $E_8 \# E_8 \# 3(S^2 \times S^2)$, we will show in \cref{lemma:homeomorphic-construction} that $M$ is homeomorphic to $M'$. But by Rochlin's theorem, there is no such smooth splitting of $K_3$. This is used in the proof of \cref{thm:some-stably-exotic-fillings-exist}, for $Y$ satisfying \eqref{main-thm-item-i}, \eqref{main-thm-item-ii}, or \eqref{main-thm-item-iii}, and  well-chosen fillings $Z$, to show that the union $M \cup_Y M'$ is nontrivial when considered in the kernel $\ker(\Omega_4(\xi_G) \to \Omega_4(\xi^{\Top}_G))$ of the forgetful map between smooth and topological bordism groups over certain fibrations, defined in \cref{section:norma-1-type-data} and dependent on $G=\pi_1(Z)$. Using a result of Kreck~\cite{surgeryandduality}*{Corollary~3} (see \cref{thm:kreckstable1type} for details), we show how unions in this kernel give rise to the eponymous \emph{stably exotic fillings} of $Y$, namely fillings that are stably homeomorphic but not stably diffeomorphic, both rel.~boundary. Working stably is thus the most natural setting in which to isolate this type of exotic behaviour, arising from Rochlin's theorem, since it factors out the effect of gauge theory. 

The above strategy does not always produce stably exotic fillings, and indeed not all 3-manifolds admit them. 

\begin{theorem}\label{thm:some-stably-exotic-fillings-dont-exist}
   Let $Y$ be a closed, smooth 3-manifold that contains a two-sided $\RP^2$. Then $Y$ does not admit stably exotic fillings. 
\end{theorem}

\cref{thm:some-stably-exotic-fillings-exist,thm:some-stably-exotic-fillings-dont-exist} partially answer the following question, which remains open in general. 

\begin{question}\label{question:main}
Which closed, smooth 3-manifolds admit stably exotic fillings?
\end{question}

For example, we do not know whether the 3-manifold constructed in \cref{example:need-lagrangian-assumption}, a $T^2$-bundle over $S^1$, admits stably exotic fillings. 

We also remark that it is open which 3-manifolds admit exotic fillings, i.e.~fillings that are homeomorphic rel.\ boundary but not diffeomorphic rel.\ boundary; cf.~\citelist{\cite{etnye-min-mukherjee:pervious}\cite{K3}*{Problem~4.6}}. Another related open question asks which orientable 3-manifolds admit exotic orientable fillings.

\subsection*{Proof strategy}
We now elaborate further on the proofs of~\cref{thm:some-stably-exotic-fillings-exist,thm:some-stably-exotic-fillings-dont-exist}. For the argument outlined on the previous page to work, at a minimum the kernel 
of $\Omega_4(\xi_G) \to \Omega_4(\xi^{\Top}_G)$ must have a nontrivial element. Given any group $G$, together with a choice of abstract \emph{normal $1$-type data} (see ~\cref{section:norma-1-type-data}), we can hope to characterise precisely when this kernel is nontrivial, by considering differentials in the James spectral sequence computing~$\Omega_4(\xi_G)$, as in~\cite{Kasprowski-Powell}. Nontriviality of the kernel turns out to be equivalent to $[K_3] \neq 0$ in $\Omega_4(\xi_G)$. Given a $3$-manifold~$Y$, we can then seek to find, as we will do in the proof of \cref{thm:some-stably-exotic-fillings-exist}, a suitable group $G$ with normal $1$-type data compatible with $Y$, and a suitable nonorientable filling~$Z$ with~$\pi_1(Z)=G$, to realise the nontrivial element of the kernel by the manifold $M\cup_Y M'$ described above. This will prove we have constructed stably exotic fillings. In contrast, in the proof of \cref{thm:some-stably-exotic-fillings-dont-exist}, we show that in the presence of $\RP^2 \times I \subseteq Y$, for any possible $\xi_G$ we will have $[K_3]=0 \in \Omega_4(\xi_G)$.
See~\cref{section:norma-1-type-data} for more information on~$\xi_G$ and~$\xi_G^\Top$, and~\cref{section:partial-characterisation} for our compatibility condition for normal $1$-type data with~$Y$.

\subsection*{The cardinality of stably exotic fillings} To supplement the results above, within a fixed (stable) homeomorphism class of fillings, we show in \cref{prop-at-most-two} that  there are at most two stable diffeomorphism classes rel.\ boundary. So it is not possible to find any larger classes of stably exotic fillings than we find in~\cref{thm:some-stably-exotic-fillings-exist}. 
On the other hand, in \cref{prop:many} we show that modifying any given stably exotic pair by connected summing both with copies of~$S^1 \times S^3$, one can obtain infinitely many stable homeomorphism classes of pairs of stably exotic fillings for any fixed $Y$ that admits them, distinguished  by their fundamental groups. 

\subsection*{Examples}

Now we give some examples designed to illustrate the use of condition~\eqref{main-thm-item-iii} in \cref{thm:some-stably-exotic-fillings-exist} and to compare the three conditions of \cref{thm:some-stably-exotic-fillings-exist}. As mentioned earlier, \cref{thm:some-stably-exotic-fillings-exist}~\eqref{main-thm-item-iii} implies \cref{thm:some-stably-exotic-fillings-exist}~\eqref{main-thm-item-i}. 
There is some overlap between \eqref{main-thm-item-i} and \eqref{main-thm-item-ii}: for example,~\eqref{main-thm-item-i} implies~\eqref{main-thm-item-ii} when $F=S^2$. But in general both set differences between conditions~\eqref{main-thm-item-i} and~\eqref{main-thm-item-ii} are nonempty. 
Since \cref{thm:some-stably-exotic-fillings-exist}~\eqref{main-thm-item-ii} does not imply \cref{thm:some-stably-exotic-fillings-exist}~\eqref{main-thm-item-i} while \cref{thm:some-stably-exotic-fillings-exist}~\eqref{main-thm-item-iii} does imply \cref{thm:some-stably-exotic-fillings-exist}~\eqref{main-thm-item-i}, one sees that \cref{thm:some-stably-exotic-fillings-exist}~\eqref{main-thm-item-ii} does not imply \cref{thm:some-stably-exotic-fillings-exist}~\eqref{main-thm-item-iii}.

\medskip

Condition~\eqref{main-thm-item-iii} in \cref{thm:some-stably-exotic-fillings-exist} may seem difficult to check at first glance, in which case the following observations, which we use in the proofs of the upcoming examples, may help. First,~$Y$ admits a tangential $\pinplus$ structure if and only if the orientation character $w_1^Y\colon \pi_1(Y)\to\Z/2$ factors through $\Z/4$ (\cref{lemma:condition-for-pin-plus}). Second, a $\pinplus$ 3-manifold is null-bordant if and only if the spin surface dual to the orientation character is spin null-bordant (\cref{prop:Kirby-taylor}).

In the following example, we describe a family of $\pinplus$ nonorientable $3$-manifolds admitting stably exotic fillings, by applying~\cref{thm:some-stably-exotic-fillings-exist}~\eqref{main-thm-item-iii}. This demonstrates 
that there exist nonorientable 3-manifolds admitting a null-bordant tangential $\pinplus$ structure, so \cref{thm:some-stably-exotic-fillings-exist}~\eqref{main-thm-item-i} does not imply \cref{thm:some-stably-exotic-fillings-exist}~\eqref{main-thm-item-iii}.

\begin{example}\label{example:orientation-reversing-monodromy}
    Let $F$ be a closed, orientable surface and let $F \to Y \to S^1$ be a fibre bundle with orientation-reversing monodromy $\varphi \colon F \to F$. Note that $Y$ is nonorientable by construction. Suppose that the fixed subgroup $H^1(F;\Z/2)^{\varphi^*}$ is at least half-rank in $H^1(F;\Z/2)$. 
    We show in \cref{prop:lagrangian-fixed-point-stable-filling} that $Y$ admits a null-bordant $\pinplus$ structure, and hence admits stably exotic fillings by \cref{thm:some-stably-exotic-fillings-exist}~\eqref{main-thm-item-iii}. 
    We show in \cref{example:need-lagrangian-assumption} that the statement is false in general without the rank assumption on  $H^1(F;\Z/2)^{\varphi^*}$.
\end{example}

Next we give an example of a circle bundle over a surface ($\neq \RP^2$) which does not admit a tangential $\pinplus$ structure. This shows that 
 \cref{thm:some-stably-exotic-fillings-exist}~\eqref{main-thm-item-iii} does not imply \cref{thm:some-stably-exotic-fillings-exist}~\eqref{main-thm-item-ii}.
 
\begin{example}\label{example-intro-circle-bundle}
Let $F= \#^3 \RP^2$ be the closed surface of nonorientable genus three and $Y := S^1 \times F$.  In \cref{prop:circle-bundle} we show that the orientation character $w_1^Y\colon \pi_1(Y)\to\Z/2$ does not factor through $\Z/4$. 
By \cref{lemma:condition-for-pin-plus}, this shows that $Y$ does not admit a tangential $\pinplus$ structure.
\end{example}

We have therefore shown that the only logical dependency between the items of \cref{thm:some-stably-exotic-fillings-exist} is the aforementioned fact that \cref{thm:some-stably-exotic-fillings-exist}~\eqref{main-thm-item-iii} implies \cref{thm:some-stably-exotic-fillings-exist}~\eqref{main-thm-item-i}.

\subsection*{Organisation}
In \cref{section:norma-1-type-data} we introduce the key concept of normal 1-type data and explain the significance of this concept to the paper.
In \cref{section:orientable-case} we prove the orientable case of \cref{thm:some-stably-exotic-fillings-exist}.
In \cref{section:partial-characterisation} we relate the existence of stably exotic fillings to the existence of stably exotic closed 4-manifolds, under certain hypotheses. The main result is the partial characterisation in \cref{thm:main-characterisation}. In \cref{section:circle-bundles}, we apply \cref{thm:main-characterisation} to prove~\cref{thm:some-stably-exotic-fillings-exist}~\eqref{main-thm-item-ii}, the circle bundle case. In \cref{section:pin-bordism}, we prove \cref{thm:some-stably-exotic-fillings-exist}~\eqref{main-thm-item-iii}. 
In \cref{sec:examples} we provide the examples promised in \cref{example:orientation-reversing-monodromy,example-intro-circle-bundle}. 
In~\cref{section:non-existence}, we prove \cref{thm:some-stably-exotic-fillings-dont-exist}. Finally in \cref{sec:number-fillings}, we consider the number of stably exotic fillings of a given $3$-manifold.

\subsection*{Conventions}

We assume that all 3-manifolds are closed, connected, and smooth, and that all 4-manifolds are compact, connected, and smooth.

\subsection*{Acknowledgements}

We are grateful to Csaba Nagy for providing us with the proof of \cref{lemma:lagrangian-complement} that we use. 
MP was partially supported by EPSRC grant EP/T028335/2.

\subsection*{Open access}

For the purposes of open access, the authors have applied a CC-BY license to this version, and will do the same to any author--accepted manuscript arising from this submission.

\bigskip 

\section{Normal 1-type data}\label{section:norma-1-type-data}

Given a $3$-manifold $Y$, our strategy to prove that it admits stably exotic fillings requires, of course, the construction of stably homeomorphic fillings $M$ and $M'$, together with some obstruction to stable diffeomorphism. The first component will be furnished by a lemma of Kreck, detailed below as~\cref{lemma:homeomorphic-construction}. For the second component, we exploit the fact that stably homeomorphic fillings $M$ and $M'$ are stably diffeomorphic rel.~$Y$ if and only if $M\cup_Y M'$ is null-bordant over the normal $1$-type of $M$ (which is the same as the normal $1$-type of $M'$ as they are stably homeomorphic); see~\cref{thm:kreckstable1type}. We will first recall some general facts about normal $1$-types of manifolds~\cite{surgeryandduality}. We then recall some specific ideas for working with normal $1$-types of manifolds with spin universal cover, developed by Teichner~\cite{teichnerthesis}.

\subsection{Normal 1-types}

\begin{definition}
    A \emph{normal $1$-type} for a smooth manifold $W$ is a pair $(B,\xi)$, where~$B$ is a space with the homotopy type of a CW complex, $\xi\colon B\to \BO$ is a $2$-coconnected fibration with connected fibre, and there exists a $2$-connected map $\overline{\nu}_W\colon W\to B$, such that $\xi\circ\overline{\nu}_W$ represents the stable normal bundle $\nu_W \colon W \to \BO$ of $W$. Such a pair $(W,\overline{\nu}_W)$ is called a \emph{normal $1$-smoothing of $W$}.

    A \emph{topological normal $1$-type} for a manifold $W$ is a pair $(B^{\top},\xi^{\top})$, where~$B^{\top}$ is a space with the homotopy type of a CW complex, $\xi^{\top}\colon B^{\top}\to \BTop$ is a $2$-coconnected fibration with connected fibre, and there exists a $2$-connected map $\overline{\nu}^{\top}_W\colon W\to B^{\top}$, such that $\xi^{\top}\circ\overline{\nu}_W^{\top}$ represents the stable normal microbundle $\nu_W^{\top} \colon W \to \BTop$ of $W$. Such a pair $(W,\overline{\nu}_W^{\top})$ is called a \emph{topological normal~$1$-smoothing of $W$}. 
    
    Here $\BTop$ is the classifying space for $\top$, the colimit under inclusion of the groups $\top(n)$ for~$n\geq 0$, consisting of homeomorphisms of $\R^n$ preserving $0$.
\end{definition}

\begin{remark}
    It follows from the Moore--Postnikov factorisation (see e.g.~\cite{zbMATH03562120}) that every smooth manifold $W$ has a corresponding normal 1-type, and similarly every topological manifold~$W$ admits a topological normal $1$-type. This theory also guarantees uniqueness, up to fibre homotopy equivalence, of these objects.
\end{remark}

\begin{definition}\label{defn:xi-structure}
    Let $W$ be a smooth manifold and let $(B,\xi)$ be the normal $1$-type of some smooth manifold, not necessarily of $W$. A lift of the stable normal bundle of $W$ to $B$ is called a \emph{$\xi$-structure}.
\end{definition}

The relevance of the normal $1$-type technology to this article comes from the following theorem.

\begin{theorem}[\citelist{\cite{surgeryandduality}*{Corollary~3}\cite{crowleysixt}*{Lemma~2.2}}]
\label{thm:kreckstable1type}\leavevmode
  \begin{enumerate}[(i)]
        \item\label{item:kreck-thm-i}  Let $M$ and $M'$ be compact, smooth 4-manifolds with fixed identifications $\partial M = Y =\partial M'$, and the same normal 1-type $\xi \colon B \to \BO$. 
    Then $M$ and $M'$ are stably diffeomorphic rel.\ $Y$ if and only if there exist normal 1-smoothings  $\overline{\nu}_M$ and $\overline{\nu}_{M'}$, of $M$ and $M'$ respectively, that agree on $Y$, such that 
    \[[M \cup_Y M',\overline{\nu}_M \cup - \overline{\nu}_{M'}] = 0 \in \Omega_4(\xi).\]
 \item\label{item:kreck-thm-ii}
  Let $M$ and $M'$ be compact, topological 4-manifolds with fixed identifications $\partial M = Y =\partial M'$, and the same topological normal 1-type $\xi^{\Top} \colon B^{\Top} \to \BTop$.
 Then $M$ and $M'$ are stably homeomorphic rel.\ $Y$ if and only if there exist topological normal 1-smoothings  $\overline{\nu}_M^{\Top}$ and~$\overline{\nu}_{M'}^{\Top}$, of $M$ and $M'$ respectively, that agree on $Y$, such that 
    \[[M \cup_Y M',\overline{\nu}_M^{\Top} \cup - \overline{\nu}_{M'}^{\Top}] = 0 \in \Omega_4(\xi^{\Top}).\]
   \end{enumerate}
 \end{theorem}

\subsection{Normal 1-type data}\label{subsec:normal-1type-date}

When the universal cover of a manifold $W$ is spin, there exists a 2-connected classifying map $c \colon W \to B\pi_1(W)$, and unique cohomology classes $w_i^W \in H^i(\pi_1(W);\Z/2)$, for $i=1,2$, such that $c^*(w_i^W) = w_i(\nu_W)$, where $\nu_W$ denotes the stable normal bundle of $W$. In this case, the fibre homotopy type of the normal $1$-type is determined by the isomorphism class of the triple $(\pi_1(W), w_1^W, w_2^W)$, in both the smooth and topological categories~\cite{teichnerthesis}*{Theorem~2.2.1~(a)}. 

Using this, we can describe an algebraic abstraction that will permit us to consider putative $4$-manifold fillings of $3$-manifolds. The following definitions are due to Teichner~\cite{teichnerthesis}*{Theorem~2.2.1~(b)(IV)}.

\begin{definition}\label{def:key-pullback}
A triple $(G,w,v)$ consisting of a finitely presented group $G$, an element $w \in H^1(G;\Z/2)$, and an element $v \in H^2(G;\Z/2)$ is called \emph{normal 1-type data}. 

Given normal $1$-type data, we define an associated pair $(B_G,\xi_G)$, where $B_G$ is a space homotopy equivalent to a CW complex and $\xi_G\colon B_G\to \BO$ is a fibration, as follows. Represent $(w,v)$ via a map $(w,v) \colon K(G,1) \xrightarrow{} K(\Z/2,1) \times K(\Z/2,2)$. 
Let $P_2(-) \colon \BO \to K(\Z/2,1) \times K(\Z/2,2)$ denote the map from $\BO$ to its Postnikov 2-type. 
We  write $B_G$ for the homotopy fibre in 
\begin{equation}\label{eqn:defn-of-B_G}
B_G \xrightarrow{\zeta_G} \BO \times K(G,1) \xrightarrow{P_2(-) + (w \times v)} K(\Z/2,1) \times K(\Z/2,2), 
\end{equation}
with $\zeta_G$ further assumed to be a fibration. 
Here the operation $+$ uses the $H$-space structure on $K(\Z/2,1) \times K(\Z/2,2)$. 
Define the fibration
\[\xi_G:=\pr_1\circ\,\, \zeta_G\colon B_G\to \BO.\]

We similarly define an associated pair $(B^{\top}_G,\xi^\top_G)$, where $B^{\top}_G$ is a space homotopy equivalent to a CW complex and $\zeta^\top_G\colon B^{\top}_G\to \BTop$ is a fibration, as follows. Write $B^{\top}_G$ for the homotopy fibre in 
\begin{equation}\label{eqn:defn-of-B_G-top}
B^{\top}_G \xrightarrow{\zeta^{\top}_G} \BTop \times K(G,1) \xrightarrow{P_2(-) + (w \times v)} K(\Z/2,1) \times K(\Z/2,2), 
\end{equation}
with $\zeta^{\top}_G$ a fibration, again using the $H$-space structure on $K(\Z/2,1) \times K(\Z/2,2)$. 
Define the fibration
\[\xi^{\top}_G:=\pr_1\circ\,\, \zeta^{\top}_G\colon B_G\to\BTop.\]
\end{definition}

\begin{remark}\label{remark:les-homotopy-groups}
       Using the long exact sequence in homotopy groups of the fibrations in~\eqref{eqn:defn-of-B_G} and~\eqref{eqn:defn-of-B_G-top}, it is straightforward to check that $\pi_1(B_G)\cong \pi_1(B^{\top}_G)\cong G$, $\pi_2(B_G) = 0 = \pi_2(B_G^{\top})$, and that both~$\xi_G$ and $\xi^{\top}_G$ are 2-coconnected. Also, note that the map $\BO\to \BTop$ induces a forgetful map 
    \[
    \eta_G\colon B_G\to B_G^{\top}.\
    \]
\end{remark}

For a smooth manifold $W$ whose universal cover is spin, the respective normal $1$-type and topological normal $1$-type are
\[
(B_W,\xi_W)=(B_G,\xi_G)\qquad\text{and}\qquad (B^{\top}_W,\xi^{\top}_W)=(B^{\top}_G,\xi^{\top}_G),
\]
constructed using the normal 1-type data $(G:=\pi_1(W), w_1^W, w_2^W)$, as described above~\cite{teichnerthesis}*{Theorem~2.2.1~(b)(IV)}. This explains the following definition.

\begin{definition}
    Let $W$ be a smooth manifold with universal cover spin. If the abstract normal $1$-type data $(G, w, v)$ is isomorphic to $(\pi_1(W), w_1^W, w_2^W)$, we will say that~\emph{$W$ has normal $1$-type data $(G, w, v)$}.
\end{definition}

Given normal $1$-type data $(G,w,v)$,
the bordism groups $\Omega_4(\xi_G)$ and $\Omega_4(\xi^{\Top}_G)$ can in principle be computed using \emph{James spectral sequences}.
These have the form  \[E^2_{p,q}=H_p(G;(\Omega_q^{\Spin})^w) \, \Rightarrow \, \Omega_*(\xi_G)\] and
\[(E^{2}_{p,q})^{\Top}=H_p(G;(\Omega_q^{\TopSpin})^w) \, \Rightarrow \, \Omega_*(\xi^{\Top}_G),\]
for $p,q \geq 0$, respectively. 
Here the coefficients are twisted using the class $w\in H^1(G;\Z/2)$ that pulls back to the orientation character.  
For $\xi_G$ and $\xi^{\Top}_G$ orientable these were constructed by Teichner~\cite{teichner-signature}*{Section~II}.
For~$\xi_G$ and $\xi^{\Top}_G$ nonorientable, 
if there exists a vector bundle $E \to K(G,1)$ such that $w_1(E) = w \in H^1(G;\Z/2)$ and $w_2(E) = v \in H^2(G;\Z/2)$, then the adaptation of the James spectral sequence was constructed in \cite{teichner-star}*{Before~Lemma~2}. 
The general case was constructed in detail by Galvin--Teichner--Vesel\'a in~\cite{GTV-xi-fillings}*{Appendix~B}.

\begin{remark}\label{rem:3manhave1type}
    All $3$-manifolds have universal cover spin because all orientable $3$-manifolds are spin. The fillings of $3$-manifolds we will produce will have normal $1$-types corresponding to the, more algebraic, normal $1$-type data we defined above. Hence we will not produce any fillings with non-spin universal cover.  It is possible that a $3$-manifold might have a pair of exotic fillings whose universal covers are non-spin. However such a pair would be stably diffeomorphic~\citelist{\cite{Gompf84}\cite{FNOP}*{Theorem~13.3}},
    so the methods of this paper would not detect them. Hence from now on we will only consider manifolds with spin universal cover.
\end{remark}

\section{All orientable 3-manifolds admit stably exotic fillings}\label{section:orientable-case}

We begin the proof of~\cref{thm:some-stably-exotic-fillings-exist}. We start with a lemma due to Kreck~\cite{Kreck84}, which gives the key construction of potentially exotic nonorientable 4-manifolds, making use of the~$K_3$ surface. This lemma will be used in this section and in subsequent sections.

\begin{lemma}[\cite{Kreck84}]\label{lemma:homeomorphic-construction}
    Let $X$ be a compact, nonorientable 4-manifold. Let $M := X\#K_3$ and let $M' :=   X \# 11(S^2 \times S^2)$. Then $M$ and $M'$ are homeomorphic $($and hence stably homeomorphic$)$ rel.\ boundary. 
\end{lemma}

\begin{proof}
We have
\[M = X\#K_3 \cong X\# E_8 \# E_8 \# 3(S^2 \times S^2) \cong X\# E_8 \# \ol{E}_8 \# 3(S^2 \times S^2) \cong X \# 11(S^2 \times S^2) = M'.\]
Here we use the classification of closed, simply connected 4-manifolds~\cite{Freedman-82}, and the fact that indefinite, symmetric, bilinear forms are determined up to isometry by their rank, signature, and parity~\cite{Milnor-Husemoller}, to see that $K_3 \cong  E_8 \# E_8 \# 3(S^2 \times S^2)$ and that $E_8 \# \ol{E}_8 \# 3(S^2 \times S^2) \cong \# 11(S^2 \times S^2)$.    To replace one copy of $E_8$ by $\ol{E}_8$ in the central homeomorphism we use that $X$ is nonorientable, so that the orientation of the 4-ball in $X$ used for the connected sum can be reversed by an ambient isotopy. 
\end{proof}

Using this lemma and \cref{thm:kreckstable1type}, we can prove~\cref{thm:some-stably-exotic-fillings-exist}~\eqref{main-thm-item-i}.

\begin{theorem}\label{thm:orientable-case}
    Let $Y$ be a closed, orientable 3-manifold. Then $Y$ admits a pair of stably exotic fillings~$M$ and $M'$. 
\end{theorem}

\begin{proof}
Since $Y$ is orientable, it admits a compact, simply connected, spin filling $X$.
Define 
\[M := X \# (S^1 \wt{\times} S^3) \# K_3 \quad\text{ and }\quad M' := X \# (S^1 \wt{\times} S^3) \# 11(S^2 \times S^2),\] 
where $S^1 \wt{\times} S^3$ is the nonorientable $S^3$ bundle over $S^1$. 
Then $M$ and $M'$ are homeomorphic by \cref{lemma:homeomorphic-construction}, and they both have the normal $1$-type data $(\Z,w,0)$ with $0\neq w\in H^1(\Z;\Z/2) \cong \Z/2$. 

If $M$ and $M'$ were stably diffeomorphic relative to $Y$, then by~\cref{thm:kreckstable1type}~\eqref{item:kreck-thm-ii} there would exist normal $1$-smoothings $\ol{\nu}_M$ and $\ol{\nu}_{M'}$ of $M$ and $M'$ that agree on $Y$ and such that $
[M\cup_Y M', \ol{\nu}_M \cup - \ol{\nu}_{M'}] = 0\in \Omega_4(\xi_\Z)$. We show next that this is not the case. 

Let $\overline{\nu}_M$ and $\overline{\nu}_{M'}$ be arbitrary normal $1$-smoothings of $M$ and $M'$ that agree on $Y$. 
The group $\Omega_4(\xi_\Z)$ is computed by a James spectral sequence with $E^2_{p,q} \cong H_p(\Z;(\Omega_q^{\Spin})^w)$; see~\cite{teichnerthesis} and~\cite{GTV-xi-fillings}*{Appendix~B}.   
Since $H_p(\Z;(\Omega_q^{\Spin})^w)=0$ for $p>1$ or $q=3$, and $w\neq 0$,
we have that 
\[\Omega_4(\xi_\Z) \cong H_0(\Z;(\Omega_4^{\Spin})^w) \cong H_0(\Z;\Z^-)\cong \Z/2,\]
and the map $16\Z\cong \Omega_4^\spin\to \Omega_4(\xi_\Z)\cong \Z/2$ is given by $16\mapsto 1$. In particular, every element~$[N,g]$ of~$\Omega_4(\xi_\Z)$ is bordant over $\xi_\Z$ to a simply connected, spin 4-manifold $N'$ (with some $\xi_\Z$-structure, as in \cref{defn:xi-structure}), and $[N,g]$ is trivial if and only if  $\sigma(N')/16\in \Z/2$ is trivial, where $\sigma(N')$ denotes the signature. 

For the pair
$(N,g)=(M\cup_Y M', \ol{\nu}_M \cup - \ol{\nu}_{M'})$, such a simply connected~$\xi_\Z$-bordant spin $4$-manifold is provided by $N'=(X\#K_3)\cup_Y (X\#11(S^2\times S^2))$  (with some $\xi_\Z$-structure). To see this, we claim that, since $S^1\wt\times S^3$ bounds $S^1\wt \times D^4$, we have~$[S^1\wt\times S^3]=0\in \Omega_4(\xi_\Z)$ for every choice of~$\xi_{\Z}$-structure. 

Now we prove the claim. First we note that $B_\Z$ is a homotopy fibre, so it suffices to show that the maps to $\BO$ and $K(\Z,1)$ extend and the extensions are compatible in $K(\Z/2,1) \times K(\Z/2,2)$. The map to $K(\Z,1) \simeq S^1$ extends by basic obstruction theory. Specifically, we note that $S^1 \wt{\times} D^4$ is obtained from $S^1 \wt{\times} S^3$ by adding a  4-cell and a 5-cell, and $\pi_3(S^1) = 0 = \pi_4(S^1)$.
Furthermore, the stable normal bundle $\nu_{S^1\wt{\times}D^4}\colon S^1\wt{\times}D^4\to \BO$ restricts to $\nu_{S^1\wt{\times}S^3}\colon S^1\wt{\times}S^3\to \BO$ on the boundary, providing the desired extension.  
The extensions are compatible in $K(\Z/2,1) \times K(\Z/2,2)$ since~$w_1(S^1 \wt{\times} D^4) \neq 0$ and $w_2(S^1 \wt{\times} D^4)=0$.

The signature of $(X\#K_3)\cup_Y -(X\#11(S^2\times S^2))$ is the signature of $K_3$, which is $16$. Thus $
[M\cup_Y M', \ol{\nu}_M \cup - \ol{\nu}_{M'}] \neq 0\in \Omega_4(\xi_\Z)$ and $M$ is not stably diffeomorphic to $M'$ rel.\ boundary, as claimed.
\end{proof}

\section{A partial characterisation of 3-manifolds admitting stably exotic fillings}\label{section:partial-characterisation}

In order to prove \cref{thm:some-stably-exotic-fillings-exist}~\eqref{main-thm-item-ii} and \eqref{main-thm-item-iii}, first we deduce a partial characterisation of $3$-manifolds admitting stably exotic fillings. As the universal cover of an arbitrary $3$-manifold is spin, the $3$-manifold determines normal $1$-type data as described in \cref{subsec:normal-1type-date}. Now we introduce the notion of abstract normal~$1$-type data for a putative filling, whose normal~$1$-type data would interact correctly with the normal~$1$-type of $Y$.

\begin{definition}\label{def:relY}
Let $Y$ be a closed, smooth $3$-manifold.
	The quadruple $(G,w,v,\varphi)$ is called 
    \emph{$Y$-filling normal 1-type data}
    if $(G,w,v)$ comprises normal 1-type data (not necessarily for $Y$), and $\varphi\colon \pi_1(Y)\to G$ is a homomorphism such that~$\varphi^*(w) = w_1^Y \in H^1(Y;\Z/2)$ and $\varphi^*(v) = w_2^Y \in H^2(Y;\Z/2)$.
\end{definition}
In the definition above, one should think of $G$ as the fundamental group of the putative filling of~$Y$, the map $\varphi$ as the putative inclusion-induced map, and $w$ and $v$ as the putative first and second Stiefel--Whitney classes of the filling.

\begin{proposition}\label{lemma:induced-map-on-1-types}
Let $(G,w,v)$ and $(H,x,y)$ be normal 1-type data.
  Let $\phi \colon G \to H$ be a homomorphism such that $\phi^*(x) = w$ and $\phi^*(y) = v$. Then there exists a map $\phi_B \colon B_G \to B_H$ covering $\id_{\BO} \times \phi \colon \BO \times K(G,1) \to \BO \times K(H,1)$, up to homotopy. Since $\xi_H$ is a fibration, we can assume that $\xi_H\circ \phi_B=\xi_G$.
\end{proposition}

\begin{proof}
    The assumption that $\phi^*(x) = w$ and $\phi^*(y) = v$ implies that the diagram
    \[
    \begin{tikzcd}
        \BO\times K(G,1)\ar[r,"\id_{\BO}\times\phi_*"]\ar[d,"P_2(-) + (w \times v)"]&\BO\times K(H,1) \ar[d,"P_2(-) + (x \times y)"]\\
        K(\Z/2,1) \times K(\Z/2,2)\ar[r,"\id"]&K(\Z/2,1) \times K(\Z/2,2)
    \end{tikzcd}
    \]
    commutes up to homotopy. Hence there is a map $\phi_B \colon B_G \to B_H$ of the homotopy fibres covering $\id_{\BO} \times \phi \colon \BO \times K(G,1) \to \BO \times K(H,1)$ up to homotopy, as needed.
\end{proof}

\begin{remark}\label{rem:bspin-maps-to-BsubG}
When $G$ is the trivial group, $B_G\simeq \BSpin$. Hence by choosing $G$ to be the trivial group in \cref{lemma:induced-map-on-1-types}, we obtain a map $\BSpin\to B_H$ over $\BO$, for any normal $1$-type data $(H,x,y)$. This determines a map $\Omega_4^{\Spin} \to \Omega_4(B_H)$, and so every spin 4-manifold has a $\xi_H$-structure. In the proof of \cref{thm:orientable-case} we produced an instance of this map for $H=\Z$ by analysing the James spectral sequence; this method generalises to arbitrary $H$. 
\end{remark}

\begin{corollary}\label{cor:existence-varphi-B}
    Let $Y$ be a closed, smooth $3$-manifold and let $(H,x,y,\varphi)$ be $Y$-filling normal 1-type data. 
    Then there is a map $\varphi_B \colon B_Y \to B_G$ over $\BO$ inducing~$\varphi$ on fundamental groups.
Hence, for every normal $1$-smoothing $\ol{\nu}_Y$ of $Y$, we obtain a corresponding $\xi_G$-structure $\varphi_B \circ \ol{\nu}_Y \colon Y \to B_G$ inducing $\varphi$ on fundamental groups. 
\end{corollary}

\begin{proof}
    Apply~\cref{lemma:induced-map-on-1-types} with $(G,w,v)$ the normal $1$-type data of $Y$, with $\varphi$ in place of $\phi$.
\end{proof}

 \cref{cor:existence-varphi-B} produces $\xi_G$-structures $\ol{\nu}_{\varphi}$ that might satisfy the hypotheses of \eqref{item-main-characterisation-ii} in the following promised partial characterisation of $3$-manifolds admitting stably exotic fillings. 

\begin{theorem}\label{thm:main-characterisation}
Let $Y$ be a closed, smooth 3-manifold and let $(G,w,v,\varphi)$ be $Y$-filling normal 1-type data.
\begin{enumerate}[(i)]
    \item\label{item-main-characterisation-i} If $Y$ admits stably exotic fillings with normal 1-type data $(G,w,v)$, then there exist closed stably exotic $4$-manifolds with normal 1-type data $(G,w,v)$.
    \item\label{item-main-characterisation-ii}  If $\varphi$ is surjective and $[Y,\ol{\nu}_\varphi] = 0 \in \Omega_3(\xi_G)$ for some $\xi_G$-structure $\ol{\nu}_\varphi \colon Y \to B_G$ inducing $\varphi$ on fundamental groups, and there exist closed stably exotic $4$-manifolds with normal 1-type data $(G,w,v)$, then $Y$ admits stably exotic fillings with normal 1-type data $(G,w,v)$.
\end{enumerate}
\end{theorem}

We will use \cref{thm:main-characterisation} to deduce the remaining items of \cref{thm:some-stably-exotic-fillings-exist} in \cref{section:circle-bundles}, and to prove \cref{thm:some-stably-exotic-fillings-dont-exist} in \cref{section:non-existence}.

\begin{proof}
First we prove~\eqref{item-main-characterisation-i}. 
Assume that $Y$ admits stably exotic fillings $M$ and $M'$ with normal 1-type data $(G,w,v)$. 
Since $M$ and $M'$ are stably homeomorphic rel.\ $Y$, by \cref{thm:kreckstable1type}~\eqref{item:kreck-thm-i} there exist normal $1$-smoothings $\ol{\nu}^{\top}_M\colon M\to B^{\top}_G$ and $\ol{\nu}^{\top}_{M'}\colon M'\to B^{\top}_G$, such that  $\ol{\nu}^{\top}_M|_Y = \ol{\nu}^{\top}_{M'}|_Y$, and such that 
\[
[M\cup_Y M', \ol{\nu}^{\top}_M \cup - \ol{\nu}^{\top}_{M'}] = 0\in \Omega_4(\xi^{\top}_G).
\] 
Recall that the spaces $B_G$ and $B_G^{\top}$ were constructed as certain homotopy fibres in \cref{subsec:normal-1type-date}, differing only in that we used $\BO$ or $\BTop$ in the construction respectively. As such, the obstruction to lifting topological normal $1$-smoothings along the forgetful map $\eta_G\colon B_G\to B_G^{\top}$ is the obstruction to lifting the stable topological normal bundle along the forgetful map $\BO\to\BTop$, which is, by definition, the Kirby--Siebenmann invariant of the manifold. Since both $M$ and~$M'$ are smooth, this obstruction vanishes for each of~$\ol{\nu}^{\top}_M$ and $\ol{\nu}^{\top}_{M'}$, and so the lift along $\eta_G$ exists.
Let~$\ol{\nu}_M\colon M\to B_G$ and $\ol{\nu}_{M'}\colon M'\to B_G$ be the resulting normal 1-smoothings  such that~$\ol{\nu}^{\top}_M=\eta_G\circ \ol{\nu}_M$ and~$\ol{\nu}^{\top}_{M'}=\eta_G\circ \ol{\nu}_{M'}$.

Since $M$ and $M'$ are not stably diffeomorphic rel.\ $Y$,~\cref{thm:kreckstable1type}~\eqref{item:kreck-thm-ii} implies that $[M\cup_Y M', \ol{\nu}_M \cup - \ol{\nu}_{M'}]\neq 0\in\Omega_4(\xi_G)$. Hence there exists a nontrivial element in the kernel of the forgetful map $\Omega_4(\eta_G)\colon \Omega_4(\xi_G)\to\Omega_4(\xi^{\top}_G)$. 
Each element of~$\Omega_4(\xi_G)$ can be represented by a 4-manifold with normal 1-type data $(G,w,v)$, by surgery below the middle dimension.  Apply this to surger the trivial element $M \cup_Y M$ and  the nontrivial element $M \cup_Y M'$ of $\Omega_4(\xi_G)$ we just constructed to obtain 4-manifolds  with normal 1-type data $(G,w,v)$.
Automorphisms of the normal 1-type cannot relate a nontrivial element of $\Omega_4(\xi_G)$ with the trivial element. Hence by~\cite{surgeryandduality}*{Theorem~C}, there exist closed stably exotic $4$-manifolds with normal 1-type data~$(G,w,v)$. 
This completes the proof of \eqref{item-main-characterisation-i}.

Next we prove \eqref{item-main-characterisation-ii}. If $w=0$, then closed stably exotic $4$-manifolds with normal 1-type data~$(G,w,v)$ do not exist by \cite{Gompf84}. Hence we assume $w\neq 0$. By our hypothesis that~$[Y,\ol{\nu}_\varphi] = 0 \in \Omega_3(\xi_G)$ for some $\xi_G$-structure $\ol{\nu}_\varphi$ of $Y$, there exists a filling $X$ of $Y$ with a smoothing~$\ol{\nu}_X \colon X \to B_G$ lifting the stable normal bundle along $\xi_G$ and extending $\ol{\nu}_\varphi$ on~$Y$. 
Using surgery, we can assume that $\ol{\nu}_X$ is a normal 1-smoothing, i.e.\ is 2-connected (recall from \cref{remark:les-homotopy-groups} that $\pi_2(B_G)=0$).   It follows from the definition of $\xi_G$ and the assumption that $w \neq 0$ that $w_1(\nu X) \neq 0$, and hence $X$ is nonorientable.  Define
\[M:= X \# K_3 \text{ and } M' := X \# 11(S^2 \times S^2).\]
Then $M$ and $M'$ are homeomorphic by \cref{lemma:homeomorphic-construction}.

Extend $\ol{\nu}_X$ to normal 1-smoothings $\ol{\nu}_M$ and $\ol{\nu}_{M'}$ of $M$ and $M'$  respectively. For this, use the unique spin structures on $K_3$ and $11(S^2 \times S^2)$ together with \cref{rem:bspin-maps-to-BsubG} to obtain $\xi_G$-structures on these manifolds, and take the connected sum.  Again both normal 1-smoothings agree with $\ol{\nu}_\varphi$ on~$Y$. For those smoothings, \[[M \cup_Y M', \ol{\nu}_M \cup - \ol{\nu}_{M'}]=[K_3] \in \Omega_4(\xi_G),\] 
because the double $[X \cup_Y X, \ol{\nu}_X \cup - \ol{\nu}_X]$ is null-bordant, as is $\# 11(S^2 \times S^2)$. 

Now we show that $M$ and $M'$, which both have normal 1-type data $(G,w,v)$, are the desired stably exotic fillings. 
Suppose then, for a contradiction, that there exists a stable diffeomorphism from $M$ to $M'$ relative to $Y$, i.e.\ a diffeomorphism $f \colon M \# \ell(S^2 \times S^2) \to M' \# \ell(S^2 \times S^2)$ for some~$\ell \in \N_0$. We use this to prove that there are no closed, stably exotic 4-manifolds with normal 1-type data $(G,w,v)$, which contradicts the hypotheses of \eqref{item-main-characterisation-ii}, and thus completes the proof. 

By \cite{Kasprowski-Powell}*{Proposition~4.1~(ii)}, 
if $[K_3]=0 \in \Omega_4(\xi_G)$, then there are no closed, stably exotic 4-manifolds with normal 1-type data $(G,w,v)$. So we aim to prove the former statement.

We claim that $\ol{\nu}_{M' \# \ell(S^2 \times S^2)} \circ f$  and $\ol{\nu}_{M \# \ell(S^2 \times S^2)}$ are homotopic rel.\ $Y$.  For this we use that both restrict to $\ol{\nu}_\varphi$ on the boundary and that $\varphi \colon \pi_1(Y) \to G$ is surjective by assumption, which implies that the inclusion  $Y \to M \# \ell(S^2 \times S^2)$ is 1-connected. 
Consider the diagram
\[
\begin{tikzcd}[column sep = 3cm]
Y \ar[r,"\ol{\nu}_\varphi"] \ar[d] & B_G \ar[d,"\xi_G"]  \\ M \# \ell(S^2 \times S^2) \ar[r,"\nu_{M\# \ell(S^2 \times S^2)}"'] \ar[ur,dashed] & \BO.
\end{tikzcd}
\]
The maps $\ol{\nu}_{M \# \ell(S^2 \times S^2)}$ and $\ol{\nu}_{M' \# \ell(S^2 \times S^2)} \circ f$ both give examples of the diagonal dashed arrow that make the diagram commute. 
Let $F$ be the homotopy fibre of $\xi_G$. 
The obstructions to finding a homotopy rel.\ $Y$ between two given lifts of $\nu_{M \# \ell(S^2 \times S^2)}$ along $\xi_G$  lie in \[H^i(M\# \ell (S^2 \times S^2),Y; \pi_i(F)).\]  Since $\xi_G$ is 2-coconnected we have that $\pi_i(F)=0$ for $i \geq 2$. Since $Y \to  M \# \ell(S^2 \times S^2)$ is 1-connected,  the cohomology groups $H^1(M\# \ell (S^2 \times S^2),Y; \pi_1(F))$ and $H^0(M\# \ell (S^2 \times S^2),Y; \pi_0(F))$ both vanish. Hence any two lifts, in particular $\ol{\nu}_{M \# \ell(S^2 \times S^2)}$ and $\ol{\nu}_{M' \# \ell(S^2 \times S^2)} \circ f$, are homotopic rel.\ $Y$.

It follows that 
\[0 = [M \# \ell(S^2 \times S^2) \cup_Y M' \# \ell(S^2 \times S^2), \ol{\nu}_{M\# \ell(S^2 \times S^2)} \cup - \ol{\nu}_{M' \# \ell(S^2 \times S^2)}]=[K_3] \in \Omega_4(\xi_G),\] as desired.  Hence, by \cite{Kasprowski-Powell}*{Proposition~4.1},  there exist no closed stably exotic $4$-manifolds with normal 1-type data $(G,w,v)$. 
\end{proof}

\section{Some nonorientable 3-manifolds that admit stably exotic fillings}

In this section we prove \cref{thm:some-stably-exotic-fillings-exist}~\eqref{main-thm-item-ii} and~\eqref{main-thm-item-iii}, using \cref{thm:main-characterisation}~\eqref{item-main-characterisation-ii}. Recall that the latter result has two main hypotheses. One of these is the existence of closed stably exotic 4-manifolds. To arrange for this, we will leverage the following statements from \cite{Kreck84} and \cite{Kasprowski-Powell}. Recall that~$H^i(\Z/2;\Z/2) = \Z/2$ for $i=1,2$. 

\begin{theorem}[\cite{Kreck84}]\label{thm:kreck-Z-2-1-1-has-closed-stable-exotica}
There exist closed stably exotic $4$-manifolds with normal 1-type data $(\Z/2,1,1)$.    
\end{theorem}

For example, $\RP^4 \# K_3$ and $\RP^4 \# 11(S^2 \times S^2)$ are stably exotic. 

\begin{theorem}[\cite{Kasprowski-Powell}*{Theorem~A}]
\label{thm:stably-exotic-dim3}
    Let $(G,w,v)$ be normal 1-type data. If there exist closed stably exotic $4$-manifolds with normal 1-type data $(G,w,v)$, then $w\neq 0$ and $w^3=wv\in H^3(G;\Z/2)$. 
    
    If moreover $G$ has cohomological dimension at most $3$, the converse also holds, i.e.\ if $w \neq 0$ and $w^3 =wv \in H^3(G;\Z/2)$ then there exist closed stably exotic $4$-manifolds with normal 1-type data~$(G,w,v)$.
\end{theorem}

In \cref{section:circle-bundles} we prove \cref{thm:some-stably-exotic-fillings-exist}~\eqref{main-thm-item-ii}. We use \cref{thm:stably-exotic-dim3} to show the existence of closed stably exotic 4-manifolds. The key observation is that circle bundles over aspherical surfaces provide surjective maps $\varphi\colon \pi_1(Y)\to G$, for groups $G$ with cohomological dimension two, that are compatible with the normal $1$-type data in the sense of~\cref{def:relY}.

In \cref{section:pin-bordism}, we prove \cref{thm:some-stably-exotic-fillings-exist}~\eqref{main-thm-item-iii} for nonorientable $Y$, the case of orientable $Y$ having been already dealt with in \cref{section:orientable-case}.
In this case, we apply \cref{thm:kreck-Z-2-1-1-has-closed-stable-exotica} to provide the closed stably exotic 4-manifolds required for \cref{thm:main-characterisation}~\eqref{item-main-characterisation-ii}. To satisfy the other hypothesis of the latter theorem, we use the normal $1$-type data $(\Z/2,1,1)$, and connect $\Omega_3(\xi_{\Z/2})$ to $\pinplus$-bordism. 

\subsection{Circle bundles admit stably exotic fillings}\label{section:circle-bundles}

The next result proves \cref{thm:some-stably-exotic-fillings-exist}~\eqref{main-thm-item-ii}.

\begin{theorem}\label{theorem:circle-bundle-case}
    Let $Y$ be a circle bundle over a surface $\Sigma \neq \RP^2$. Then $Y$ admits stably exotic fillings. 
\end{theorem}

\begin{proof}
If $\Sigma \cong S^2$ then $Y$ is orientable, and this case was covered in \cref{thm:orientable-case}.  
Hence we assume for this proof that $\Sigma$ is aspherical.

    Let $p \colon X \to \Sigma$ be the total space of  the disc bundle over $\Sigma$ obtained from $Y$ by filling in each fibre with $D^2$. This is possible since $\Diff(S^1) \simeq \OO(2)$. Then the universal cover $\wt X \simeq \wt{\Sigma}$ is contractible, and in particular is spin. 
    Furthermore, $\pi_1(X)\cong \pi_1(\Sigma)$, and hence $\pi_1(X)$ is cohomologically $2$-dimensional. 
    Let $i\colon Y\to X$ be the inclusion map and consider the $Y$-filling normal 1-type data given by $(G,w,v,\varphi) := (\pi_1(X),w_1^X,w_2^X,\pi_1(i))$.  
    Note that $w^3=wv\in H^3(G;\Z/2)$ automatically since $H^3(G;\Z/2)=0$. Hence by \cref{thm:stably-exotic-dim3} there exist closed stably exotic $4$-manifolds with normal 1-type data $(G,w,v)$. This gives one of the hypotheses needed to apply \cref{thm:main-characterisation}~\eqref{item-main-characterisation-ii}.

The map $\varphi \colon \pi_1(Y) \to G=\pi_1(X)$ is surjective by construction. The $4$-manifold $X$  shows that~$Y$ bounds over the normal 1-type. Thus all the hypotheses \cref{thm:main-characterisation}~\eqref{item-main-characterisation-ii} are satisfied, and so~$Y$ admits stably exotic fillings. 
\end{proof}

\subsection{Pin bordism and stably exotic fillings}\label{section:pin-bordism} 
We will soon prove \cref{thm:some-stably-exotic-fillings-exist}~\eqref{main-thm-item-iii}. To do so, we will require the following facts about Pin structures and Pin bordism groups from Kirby--Taylor~\cite{kirby-taylor:pin}. 
Let $Y$ be a manifold with stable tangent bundle~$TY$ and stable normal bundle~$\nu Y$. 

\begin{proposition}[{\cite{kirby-taylor:pin}*{Lemma~1.7}}]Let $Y$ be a manifold. Let $\zeta$ denote a $($stable$)$ vector bundle over $Y$ with determinant bundle $\det(\zeta)$. 

\begin{enumerate}[(i)]
\item A $\pinplus$ structure on $\zeta$ is equivalent to a spin structure on $\zeta \oplus 3\det(\zeta)$. Such a structure exists if and only if $w_2(\zeta)=0$.  

\item A $\pinminus$ structure on $\zeta$ is equivalent to a spin structure on $\zeta \oplus \det(\zeta)$.  Such a structure exists if and only if $w_2(\zeta) + w_1(\zeta)^2=0$. 
\end{enumerate}
\end{proposition}

Here is an immediate corollary.

\begin{corollary}\label{cor:equiv-pinplus-minus}
Let $Y$ be a manifold with stable tangent bundle~$TY$ and stable normal bundle~$\nu Y$. 
\begin{enumerate}[(i)]
    \item\label{item:cor-equiv-pinplus-minus-i} The manifold $Y$ admits a tangential $\pinplus$ structure if and only if $w_2(TY)=0$, if and only if $w_2(\nu Y) + w_1(\nu Y)^2=0$, if and only if $Y$ admits an normal $\pinminus$ structure.  
\item\label{item:cor-equiv-pinplus-minus-ii} The manifold $Y$ admits a tangential $\pinminus$ structure if and only if $w_2(TY)+ w_1(TY)^2=0$, if and only if $w_2(\nu Y)=0$, if and only if $Y$ admits an normal $\pinplus$ structure. 
\end{enumerate}
\end{corollary}

\begin{remark}
    The usual convention is to write $w_i(Y):=w_i(TY)$, but since the normal bundle also plays a large r\^{o}le in this article we preserve the $T$ in the notation to avoid potential confusion. 
\end{remark}

\begin{proposition}[{\cite{kirby-taylor:pin}*{\S 2}}]\label{lemma-existence-of-pin-plus-structure}
Every 3-manifold $Y$ has a tangential $\pinminus$ structure, and it has a tangential $\pinplus$ structure if and only if $w_1^2(TY)=0$. 
\end{proposition}

\begin{proof}
Details were not given in \cite{kirby-taylor:pin}, so we provide them here. In this proof, for brevity, denote~$v_i:=v_i(TY)$ and $w_i:=w_i(TY)$ for the Wu and Stiefel--Whitney classes respectively. We similarly omit $TY$ from the notation for the total Wu and Stiefel--Whitney classes.

First, we show that $v_2=0$. By definition of the Wu class, we have that $v_2 \cup x = \Sq^2(x)$ for every $x \in H^1(Y;\Z/2)$, and $\Sq^2(x) =0$ because $\Sq^2$ always vanishes on $H^1$. By Poincar\'{e} duality it follows that $v_2=0$ as claimed.  

The Wu formula for the total classes is $w = \Sq(v)$. From this we can compute that $v_1=w_1$ and~$w_2 = v_2 + \Sq^1 (v_1) = v_2 + w_1^2$, hence $v_2 = w_2 + w_1^2$, as we are working over $\Z/2$.  Now since~$v_2=0$ we must have $w_2 + w_1^2=0$, so every 3-manifold admits a tangential $\pinminus$ structure. Also $w_2=0$ if and only if $w_1^2=0$, so a 3-manifold admits a tangential $\pinplus$ 
structure if and only if $w_1^2=0$. 
\end{proof}

\begin{lemma}\label{lemma:w_1-squared-zero-iff-Z4-lift}
Let $Y$ be a manifold. 
We have that $w_1^2(TY)=0$
if and only if $w_1(TY)\colon \pi_1(Y)\to \Z/2$ admits a lift to $\Z/4$ along the surjective map $\Z/4\onto \Z/2$.    
\end{lemma}

\begin{proof}
The exact sequence $0 \to \Z/2\to \Z/4\to \Z/2 \to 0$ yields an exact sequence 
\[H^1(Y;\Z/4)\to H^1(Y;\Z/2)\xrightarrow{\beta} H^2(Y;\Z/2),\] 
where $\beta$ denotes the Bockstein homomorphism. It is a standard property of Steenrod squares (see e.g.~\cite{hatcher-AT}*{Section~4L}) that $\beta=\Sq^1$. By the universal coefficient theorem and since $\Z/4$ and~$\Z/2$ are abelian, we thus have the exact sequence 
\[
\begin{tikzcd}
    \Hom(\pi_1(Y),\Z/4)\ar[r]   &\Hom(\pi_1(Y),\Z/2)\ar[r,"\beta=\Sq^1"]  &H^2(Y;\Z/2).
\end{tikzcd}
\]
We see from the sequence that $w_1(TY)\in \Hom(\pi_1(Y),\Z/2)$ lifts to a map to $\Z/4$ if and only if $w_1^2(TY)=\Sq^1(w_1(TY))=\beta(w_1(TY))=0$.
\end{proof}

We deduce the following characterisation for when tangential $\pinplus$ structures exist on $3$-manifolds.

\begin{proposition}\label{lemma:condition-for-pin-plus}
A closed, nonorientable 3-manifold $Y$ admits a tangential $\pinplus$ structure if and only if the orientation character $w \colon \pi_1(Y) \onto \Z/2$ factors as $\pi_1(Y) \to \Z/4 \twoheadrightarrow \Z/2$. 
\end{proposition}

\begin{proof}
     Combine~\cref{lemma-existence-of-pin-plus-structure,lemma:w_1-squared-zero-iff-Z4-lift}.
\end{proof}

This will be used in \cref{sec:examples}, when we seek to apply \cref{thm:pinplus-summary} below. 
The following is also from Kirby--Taylor \cite{kirby-taylor:pin}.

\begin{proposition}[\cite{kirby-taylor:pin}*{Theorems~5.1 and 5.2}]
\label{prop:Kirby-taylor}
\leavevmode
\begin{enumerate}[(i)]
\item\label{item:kirby-taylor-i}
There are isomorphisms $\Omega_3^\pinplus\xrightarrow{\cong}\Omega_2^\spin \cong \Z/2$, given by taking a spin surface dual to $w_1$, and then taking the Arf invariant of its associated quadratic form. 
\item\label{item:kirby-taylor-ii}
There is an isomorphism $\Omega^{\pinplus}_4 \cong \Z/16$, under which $[K3]$ maps to $8 \in \Z/16$.
\end{enumerate}
\end{proposition}

\begin{remark}
It is also true that $\Omega^{\pinminus}_4=0$. For the purposes of this paper, this explains why it is not so helpful to look at $\pinminus$  structures, even though every 3-manifold admits a $\pinminus$ structure. 
\end{remark}

The following theorem  proves \cref{thm:some-stably-exotic-fillings-exist}~\eqref{main-thm-item-iii}.

\begin{theorem}\label{thm:pinplus-summary}
    Let $Y$ be a closed 3-manifold. 
    If $Y$ admits a null-bordant tangential $\pinplus$ structure, then $Y$ admits stably exotic fillings. 
\end{theorem}

\begin{proof}
We already proved \cref{thm:pinplus-summary} in the case that $Y$ is orientable, in \cref{thm:orientable-case}. Hence we can restrict here to the case that $Y$ is nonorientable. 

Let $\varphi\colon \pi_1(Y)\onto \Z/2$ be the orientation character of $Y$. Here $\varphi$ is surjective since $Y$ is nonorientable. Let $\xi_{\Z/2}$ be the normal $1$-type for the normal 1-type data $(\Z/2,w=1,v=1)$, constructed as described in \cref{subsec:normal-1type-date}. Here $w \in H^1(\Z/2;\Z/2)\cong \Z/2$ and $v \in H^2(\Z/2;\Z/2)\cong \Z/2$ are the generators. 

We claim that $[Y,\overline{\nu}_\varphi]=0\in\Omega_3(\xi_{\Z/2})$ for some $\xi_{\Z/2}$-structure $\overline{\nu}_\varphi$ inducing $\varphi \colon \pi_1(Y) \to \pi_1(B_{\Z/2}) = \Z/2$. Indeed, let~$X$ be a (hypothesised) tangential $\pinplus$ filling of $Y$, and let $w_1^X \colon X \to K(\Z/2,1)$ be the map corresponding to the orientation character of $X$, which is such that the composition 
\begin{equation}\label{eqn:composition}
\pi_1(Y) \to \pi_1(X) \xrightarrow{(w_1^X)_*} \Z/2    
\end{equation} 
equals $\varphi$.  
Since $X$ has a tangential $\pinplus$ structure, $\wt X$ is spin and $w_1(\nu X)^2 + w_2(\nu X) = 0$ by \cref{cor:equiv-pinplus-minus}~\eqref{item:cor-equiv-pinplus-minus-i}. Using the cup product structure on $K(\Z/2,1)\simeq\RP^\infty$, naturality, the fact that~$w$ pulls back to $w_1(\nu X)$, and finally the previous sentence, we therefore have \[(w_1^X)^*(v) = (w_1^X)^*(w^2) = (w_1^X)^*(w)^2 = w_1(\nu X)^2 = w_2(\nu X) \in H^2(X;\Z/2).\] 
By \cref{lemma:induced-map-on-1-types} we obtain a map $(w_1^X)_B \colon B_{\pi_1(X)} \to B_{\Z/2}$ over $\BO$ that induces $(w_1^X)_*$ on fundamental groups. Thus $X$ admits a $\xi_{\Z/2}$-structure $(w_1^X)_B\circ \ol{\nu}_X$, where $\ol{\nu}_X$ is a normal $1$-smoothing of $X$. This $\xi_{\Z/2}$-structure on $X$ induces a $\xi_{\Z/2}$-structure $\ol{\nu}_\varphi$ on $Y$ that gives $\varphi$ on fundamental groups by \eqref{eqn:composition}. We see that $X$ is a null-bordism witnessing that $[Y,\ol{\nu}_\varphi] = 0 \in \Omega_3(\xi_{\Z/2})$. This completes the proof of the claim. 

We know from \cref{thm:kreck-Z-2-1-1-has-closed-stable-exotica}  that the normal 1-type data $(\Z/2,1,1)$ admits closed stably exotic 4-manifolds. 
Hence we may apply \cref{thm:main-characterisation}~\eqref{item-main-characterisation-ii} to see that $Y$ admits stably exotic fillings.
\end{proof}

\begin{remark}
    There is a free transitive action of $H^1(Y;\Z/2)$ on the set of $\pinplus$ structures, and in applications of \cref{thm:pinplus-summary},  once we know that $Y$ admits a $\pinplus$ structure, we are free to try to modify it by this action to arrange for $Y$ to be $\pinplus$-null-bordant. 
\end{remark}

\section{Examples}\label{sec:examples}

We provide the $\pinplus$ examples promised in the introduction. 
The first aim is to prove that the examples in \cref{example:orientation-reversing-monodromy} behave as asserted. 
For this we need a couple of lemmas, before \cref{prop:lagrangian-fixed-point-stable-filling}, which is the main result of this section. 

First, we set up some notation. 
Let $F$ be a closed, orientable surface of genus $g$. 
For a $\Z/2$ vector space $P$, we write $P^* := \Hom_{\Z/2}(P,\Z/2)$ for the dual vector space. 
For any $i \geq 0$, let $\ev \colon H^i(F;\Z/2) \xrightarrow{\cong} H_i(F;\Z/2)^*$ denote the evaluation map. 
We will consider the $\Z/2$-valued \emph{intersection form}
\[\lambda_F^{\Z/2} \colon H_1(F;\Z/2) \times H_1(F;\Z/2) \to \Z/2; \quad (x,y) \mapsto \big(\ev \circ PD^{-1}(y)\big) (x)\] 
and the Poincar\'{e} dual \emph{cup product form}
\[- \smile - \colon H^1(F;\Z/2) \times H^1(F;\Z/2) \to \Z/2; \quad (f,h) \mapsto \ev(f \smile h) [F]_{\Z/2}.\]
Both can be represented by the diagonal sum of $g$ copies of the standard symplectic matrix $\bsm 0 & 1 \\  1 & 0  \esm$. 
A \emph{Lagrangian} for either pairing is a direct summand $L$ with $L = L^{\perp}$. The proof we give of the next lemma was suggested to us in this formulation by Csaba Nagy; it is closely related to the proof of \cite{Nagy-2024}*{Theorem~7.6}. 

\begin{lemma}\label{lemma:lagrangian-complement}
    Let $(P,\psi)$ be a hyperbolic form over $\Z/2$ of dimension $2g$ and let $V$ be a subspace of $P$ of dimension at least $g$. Then there is a Lagrangian $K \subseteq P$ such that $K +V = P$. 
\end{lemma}

\begin{proof}
By passing to a $g$-dimensional subspace of $V$, without loss of generality we can assume that~$V$ is rank $g$. Writing $j\colon V\to P$ for the inclusion, we have that $V^\perp$ is the kernel of the surjective map $j^*\circ \mathrm{ad}(\psi)\colon P\to V^*$, so  $\dim V^\perp=\dim P-\dim V^*=g = \dim V$. Let $Z:=V\cap V^\perp$ and choose direct sum decompositions $V=U\oplus Z$ and $V^\perp=U'\oplus Z$, so that $V+V^\perp=Z\oplus U\oplus U'$. Noting that $Z$ is the radical for the restriction of $\psi$ to both $V$ and to $V^\perp$, we have that $\psi$ restricts to a nonsingular form on both $U$ and $U'$. As $U$ and $U'$ have the same dimension, we may thus choose an isomorphism of nonsingular alternating forms $f\colon (U,\psi|_U)\to (U\,\psi|_{U'})$. Such an isomorphism determines a Lagrangian $\Delta_f:=\{(y,f(y))\,|\,y\in U\}\subseteq U\oplus U'$.  By the same argument as the second sentence of the proof, $\dim(U \oplus U')^{\perp} = \dim P - \dim(U \oplus U')$. 
The restriction of $\psi$ to $U\oplus U'$ is nonsingular, so $(U\oplus U')\cap (U\oplus U')^\perp=\{0\}$. Combining the two previous sentences, we have 
\begin{equation}\label{eqn:decomposition-of-P}
P=(U\oplus U')\oplus (U\oplus U')^\perp.
\end{equation}
We will now extend $\Delta_f$ to a full Lagrangian of $(P,\psi)$. Note that $Z\subseteq (U\oplus U')^\perp$. 
Since $\dim U + \dim Z = \dim V = g = \dim V^{\perp} = \dim U' + \dim Z$ we have that $2 \dim Z = 2g-\dim U - \dim U'$. Also $2g = \dim P = \dim U + \dim U' + \dim (U\oplus U')^{\perp}$, and hence $\dim (U\oplus U')^{\perp} = 2 \dim Z$, i.e.\ $Z\subseteq (U\oplus U')^\perp$ is half-rank.  Moreover we already noted that $\psi$ vanishes identically on $Z$, so $Z$ is isotropic.  
Since the decomposition in \eqref{eqn:decomposition-of-P} is by definition orthogonal with respect to $\psi$, and $\psi$ is nonsingular, $\psi$ restricts to a nonsingular form on $(U\oplus U')^\perp$.  It follows that $Z$ is a Lagrangian for the restriction in $(U\oplus U')^\perp$. Moreover $Z$ admits a complementary Lagrangian $L$,  so that $Z\oplus L= (U\oplus U')^\perp$.  
To find one,  start with a complementary summand $L'$. Since $\psi|$ is nonsingular, $Z \to (L')^*$ is an isomorphism, and hence by adding elements of $Z$ to the elements of a basis for~$L'$, we can change $L'$ to a Lagrangian $L$ that is still complementary to $Z$. 

For any complementary Lagrangian $L$ to $Z$ in $(U\oplus U')^\perp$, the submodule
\[
K:=\Delta_f\oplus L\subseteq (U\oplus U')\oplus (U\oplus U')^\perp
\]
is a Lagrangian, and is such that $K+V=P$.
\end{proof}

For a closed, orientable surface $F$, let \[\operatorname{ad}(- \smile -) \colon H^1(F;\Z/2) \to H^1(F;\Z/2)^*;\; \; f \mapsto (h \mapsto \ev(f \smile h) [F]_{\Z/2})\] denote the adjoint of the cup product form. Equivalently, $f \mapsto h \frown (f \frown [F]_{\Z/2})$. 

\begin{lemma}\label{lemma:surjective-to-dual-of-lagrangian}
     Let $F$ be a closed, orientable surface of genus $g$. Let $V$ be a summand of $H^1(F;\Z/2)  \cong (\Z/2)^{2g}$ whose rank is at least $g$. Then there is a Lagrangian $K \subseteq H^1(F;\Z/2)$ of the cup product form such that 
     \[V \into H^1(F;\Z/2) \xrightarrow{\operatorname{ad}(- \smile -)} H^1(F;\Z/2)^* \to K^*\]
     is surjective. 
\end{lemma}

\begin{proof}
By \cref{lemma:lagrangian-complement}, there exists a Lagrangian $K$ for the cup product form such that $K +V = H^1(F;\Z/2)$. Since $K$ is a Lagrangian, $K \subseteq K^{\perp}$ and hence the composition
\begin{equation}\label{eqn:lagrangian}
    K \into H^1(F;\Z/2) \xrightarrow{\operatorname{ad}(- \smile -)} H^1(F;\Z/2)^* \to K^*
\end{equation}
is the zero map.  We develop the diagram
\[\begin{tikzcd}
 & K + V \ar[d,"\cong"] \ar[rr,twoheadrightarrow] & &   (K + V)/K \ar[d] \\
   V \ar[r] \ar[ur]  & H^1(F;\Z/2) \ar[r,"\cong"] &  H^1(F;\Z/2)^* \ar[r,twoheadrightarrow] & K^*.
\end{tikzcd}\]
The bottom row is the composition in the statement of the lemma. 
The down-right-right composition $K +V \to K^*$ factors through $(K+V)/K$ by \eqref{eqn:lagrangian}, which defines the right vertical map.  
We have labelled several maps as being surjective or isomorphisms. We see that the composition $K+V \to K^*$ is surjective. It follows that  the right vertical map $(K + V)/K \to K^*$ is surjective. Moreover, observe that the upper composition $V \to K+V \to (K + V)/K$ is surjective. It follows that the lower composition $V \to K^*$ is surjective, as desired. 
\end{proof}

\begin{corollary}\label{cor:surjective-to-dual-of-lagrangian}
     Let $F$ be a closed, orientable surface of genus $g$. Let $V$ be a summand of $H^1(F;\Z/2) \cong (\Z/2)^{2g}$ of rank at least $g$. Then there is a Lagrangian $L \subseteq H_1(F;\Z/2)$ of the intersection form such that 
     \[V \into H^1(F;\Z/2) \xrightarrow{\ev} H_1(F;\Z/2)^* \to L^*\]
     is surjective. 
\end{corollary}

\begin{proof}
Let $K \subseteq H^1(F;\Z/2)$ be the Lagrangian as in \cref{lemma:surjective-to-dual-of-lagrangian}, and define $L := PD(K) \subseteq H_1(F;\Z/2)$. Since Poincar\'{e} duality determines an isometry from $-\smile -$ to $\lambda_{\Z/2}$, $L$ is a Lagrangian of $\lambda_{\Z/2}$. We have the following diagram. It is straightforward to check that it commutes. 
\[\begin{tikzcd}
    V \ar[r,hook] & H^1(F;\Z/2) \ar[r,"\ev"] \ar[d,"PD","\cong"']  & H_1(F;\Z/2)^*  \ar[d,"PD^*","\cong"'] \ar[r] & L^* \ar[d,"PD^*|_{L^*}","\cong"']  \\
& H_1(F;\Z/2) \ar[r,"\ev"] & H^1(F;\Z/2)^* \ar[r] & K^*.
\end{tikzcd}\]
The down-right composition $H^1(F;\Z/2) \to H^1(F;\Z/2)^*$ sends \[f \mapsto  \big(h \mapsto h(f \frown [F]_{\Z/2})\big) = \big(h \mapsto h \frown (f \frown [F]_{\Z/2})\big) = \big(h \mapsto \ev(h \smile f)[F]_{\Z/2}\big),\]
and so is $\operatorname{ad}(- \smile -)$. Thus $V$ surjects onto $K^*$ in this diagram by \cref{lemma:surjective-to-dual-of-lagrangian}, and hence by commutativity $V$ also surjects onto $L^*$. 
\end{proof}

We are now ready to provide the examples promised in \cref{example:orientation-reversing-monodromy}.

 \begin{proposition}\label{prop:lagrangian-fixed-point-stable-filling}
    Let $F$ be a closed, orientable surface of genus $g$ and let $F \to Y \to S^1$ be a fibre bundle with orientation-reversing monodromy $\varphi \colon F \to F$. Suppose that the fixed subgroup $H^1(F;\Z/2)^{\varphi^*} \subseteq H^1(F;\Z/2)$ has rank at least $g$. 
    Then $Y$ admits a null-bordant $\pinplus$ structure, and hence admits stably exotic fillings by \cref{thm:some-stably-exotic-fillings-exist}~\eqref{main-thm-item-iii}.
 \end{proposition}

\begin{proof} 
    The orientation character factors through the bundle projection map as
    \[\pi_1(Y) \to \pi_1(S^1) \cong \Z \twoheadrightarrow \Z/2.\]
    Hence there is a lift to $\Z/4$ and so there is a tangential $\pinplus$ structure on $Y$ by \cref{lemma:condition-for-pin-plus}. 
    A copy of the fibre $F$ is dual to~$w_1$. This has an induced spin structure and, by \cref{prop:Kirby-taylor}~\eqref{item:kirby-taylor-i},~$Y$ is $\pinplus$ null-bordant if and only if $F$ is spin null-bordant. To arrange the latter, we are free to change the $\pinplus$ structure on $Y$, and hence the induced spin structure on $F$, using the action of~$H^1(Y;\Z/2)$. 

We have a Wang exact sequence
\begin{equation}\label{eq:wang}
H^0(F;\Z/2) \xrightarrow{0} H^0(F;\Z/2) \to H^1(Y;\Z/2) \to H^1(F;\Z/2) \xrightarrow{\varphi^* - \Id} H^1(F;\Z/2)
\end{equation}
leading to a short exact sequence
\[0 \to \Z/2 \to H^1(Y;\Z/2) \to H^1(F;\Z/2)^{\varphi^*} \to 0,\]
where $H^1(F;\Z/2)^{\varphi^*}$ denotes the fixed points under $\varphi^*$.

Let $\iota \colon F \to Y$ be the inclusion. If  $\im\big(\iota^* \colon H^1(Y;\Z/2) \to H^1(F;\Z/2)\big) = H^1(F;\Z/2)^{\varphi^*} \subseteq H^1(F;\Z/2)$ is at least half-rank, then by \cref{cor:surjective-to-dual-of-lagrangian}, there is a Lagrangian $L$ for the $\Z/2$-intersection form such that 
\[H^1(F;\Z/2)^{\varphi^*} \to H^1(F;\Z/2) \xrightarrow{\ev} H_1(F;\Z/2)^* \to L^*\]
     is surjective. 
    The group $H^1(Y;\Z/2)$ acts transitively on the $\pinplus$ structures on $Y$. The effect of the action of $x \in H^1(Y;\Z/2)$ on a spin structure of $F$ is the effect of $\iota^*(x)$. We can therefore make any change of spin structures on $F$ that corresponds to the action of an element of $ \im \iota^* = H^1(F;\Z/2)^{\varphi^*}$. Under the action, $\iota^*(x)$ changes the spin structure on a curve $\gamma \subseteq F$ if and only if $\ev \circ \iota^*(x)[\gamma] =1 \in \Z/2$. Since $H^1(F;\Z/2)^{\varphi^*}$  maps surjectively to $L^*$, it follows that by changes of the $\pinplus$ structure on $Y$, we can change the spin structure induced on $F$ to realise any desired spin structure on the Lagrangian $L$. Specifically, we make the spin structure into the bounding one on $L$. After that, $F$ spin bounds, and hence $Y$ admits a  $\pinplus$ filling by \cref{prop:Kirby-taylor}~\eqref{item:kirby-taylor-i}.  
\end{proof}

Next  we  give an example, also promised in \cref{example:orientation-reversing-monodromy}, demonstrating that   \cref{prop:lagrangian-fixed-point-stable-filling} is false in general without the hypothesis that $H^1(F;\Z/2)^{\varphi^*}$ is at least half-rank. 

\begin{example}\label{example:need-lagrangian-assumption}
 Let $F=T^2$ be a torus. Let $\varphi \colon F \to F$ be a diffeomorphism whose induced map on $H^1(F;\Z/2)$ is represented by $\left(\begin{smallmatrix}
     1 & 1 \\ 1 & 0 
 \end{smallmatrix} \right)$, and consider the associated $T^2$-bundle $Y$ over $S^1$.  
 The monodromy $\varphi$ is orientation-reversing, and so $Y$ is nonorientable.
As in the proof of \cref{prop:lagrangian-fixed-point-stable-filling}, the orientation character factors through the bundle projection map as
    $w^Y \colon \pi_1(Y) \to \pi_1(S^1) \cong \Z \twoheadrightarrow \Z/2$, so there is a lift to $\Z/4$ and hence $Y$ admits a tangential $\pinplus$ structure by~\cref{lemma:condition-for-pin-plus}. 
 
 One computes that~$\varphi^*$ acts transitively on the nonzero elements of $H^1(F;\Z/2)$, and hence has no nonzero fixed points; in particular the rank of the fixed points is $0 < g/2 = 1/2$. By the Wang sequence~\eqref{eq:wang}, the map $H^1(Y;\Z/2) \to H^1(F;\Z/2)$ is trivial, and hence every $\pinplus$ structure on $Y$ induces the same spin structure on~$F$. We need to determine the Arf invariant of this spin structure to identify its spin bordism class.  
 Since the action of $\varphi^*$ on  $H^1(F;\Z/2) \sm \{0\}$ is transitive, the only spin structure that is invariant under $\varphi^*$ is the non-bounding structure on curves corresponding to all three nontrivial elements $(1,0)$, $(0,1)$, and $(1,1)$ in $(\Z/2)^2 \cong H^1(F;\Z/2)$. To see this, note that all spin structures restrict to a non-bounding one on at least one of these three curves.
This spin structure has Arf invariant nonzero, and hence as a spin surface $F$ is nontrivial in $\Omega_2^{\Spin}$. So $Y$ does not $\pinplus$-bound, by \cref{prop:Kirby-taylor}~\eqref{item:kirby-taylor-i}. Thus the assumption on the fixed points was necessary in \cref{prop:lagrangian-fixed-point-stable-filling}.   We do not know whether this $Y$ admits stably exotic fillings, for some other normal 1-type data. 
\end{example}

Finally we construct a circle bundle over a surface ($\neq \RP^2$) which does not admit a tangential~$\pinplus$ structure, as mentioned in \cref{example-intro-circle-bundle}. 

\begin{proposition}\label{prop:circle-bundle}
    Let $F= \#^3 \RP^2$ be the closed surface of nonorientable genus three and let $Y := S^1 \times F$. The orientation character $w_1^Y\colon \pi_1(Y)\to\Z/2$ does not factor through $\Z/4$, and hence~$Y$ does not admit a tangential $\pinplus$ by \cref{lemma:condition-for-pin-plus}.
\end{proposition}

\begin{proof}
    The orientation character of $F$ factors as 
    \[w_1^{F} \colon \pi_1(F) \to \pi_1(F) \times \pi_1(S^1) \xrightarrow{\cong} \pi_1(Y) \xrightarrow{w^Y_1} \Z/2.\] 
    Thus if $w_1^Y$ factors through $\Z/4$ then so does $w_1^{F}$.  
    Write $\pi_1(F) = \pi_1(\#^3 \RP^2) \cong \langle a,b,c \,|\, a^2b^2c^2 \rangle$, where the orientation character sends each of $a,b$, and~$c$ to $1$. The orientation character of $F$ factors further through the Hurewicz map as
    \[
    w_1^{F} \colon \pi_1(F) \to H_1(F;\Z) \xrightarrow{\overline{w}_1^{F}} \Z/2.
    \]
    If $w_1^{F}$ factors through $\Z/4$, so does $\overline{w}_1^{F}$; since $\Z/4$ is abelian, maps to it factor through the abelianisation.  Suppose there is such a factorisation $H_1(F;\Z)\xrightarrow{\varphi} \Z/4\to \Z/2$. The element $[abc] \in H_1(F;\Z)$ has order two, so we must have $\varphi([abc])=2\in \Z/4$, which maps trivially to~$\Z/2$. This contradicts the fact that $\overline{w}_1^{F}([abc])=1\in \Z/2$, which holds since each of $a$, $b$, and $c$ are orientation-reversing. 
    Thus $\overline{w}_1^{F}$ does not factor through $\Z/4$, and hence neither does $w_1^{F}$.  
\end{proof}

\section{No stably exotic fillings in the presence of a two-sided projective plane}\label{section:non-existence}

We prove \cref{thm:some-stably-exotic-fillings-dont-exist}, whose statement we recall. 

\begin{reptheorem}{thm:some-stably-exotic-fillings-dont-exist}
   Let $Y$ be a closed 3-manifold that contains a two-sided $\RP^2$.  Then $Y$ does not admit stably exotic fillings. 
\end{reptheorem}

\begin{proof}
	Let $X$ be a smooth filling of $Y$, which exists since the smooth bordism group of unoriented 3-manifolds is trivial. If the universal cover $\wt{X}$ is non-spin, then by \citelist{\cite{Gompf84}\cite{FNOP}*{Theorem~13.3}}, any 4-manifold stably homeomorphic to $X$ is stably diffeomorphic to $X$, and we are done. So we may assume that  $\wt{X}$ is spin. 
    
    Let $(G,w,v)$ be the normal 1-type data of $X$.
    Then by hypothesis we have maps $\RP^2\xrightarrow{f}Y\xrightarrow{\iota} X$, where $\iota\colon Y \to X$ is the inclusion map. 
 Since $f(\RP^2) \subseteq Y$ and $\iota(Y) \subseteq X$ have trivial normal bundles, $(\iota \circ f)^*w_i(\nu_{X}) = w_i(\nu_{\RP^2})$ for $i=1,2$. 
  Let $w_i^{\RP^2} \in H^i(\Z/2;\Z/2)$ denote classes pulling back to $w_i(\nu_{\RP^2})$ under the classifying map $c \colon \RP^2 \to K(\Z/2,1)$, for $i=1,2$.  Since $w(\nu_{\RP^2})=1+x$, with $x\in H^1(\RP^2;\Z/2)$ the generator,  we see that $w_1^{\RP^2} \neq 0$ and  $w_2^{\RP^2}=0$. 
  Consider the composition 
  \[g\colon \Z/2\xrightarrow[\cong]{c_*^{-1}}\pi_1(\RP^2)\xrightarrow{(\iota \circ f)_*}\pi_1(X)\xrightarrow{\cong}G.\]
Denoting the induced map $K(\Z/2,1) \to K(G,1)$ also by $g$, we have a commutative diagram
\[\begin{tikzcd}
    \RP^2 \ar[r,"\iota \circ f"] \ar[d,"c"] & X \ar[d,"c_X"] \\ K(\Z/2,1) \ar[r,"g"] & K(G,1). 
\end{tikzcd}\]
The induced square in cohomology for $i=1,2$, is: 
\[\begin{tikzcd}[column sep = large]
    H^i(\RP^2;\Z/2) &   H^i(X ;\Z/2) \ar[l,"(\iota \circ f)^*"'] \\ H^i(\Z/2 ;\Z/2) \ar[u,"c^*","\cong"']  & H^i(G ;\Z/2). \ar[l,"g^*"'] \ar[u,"c_X^*"]
\end{tikzcd}\]
 By commutativity of this square, and the fact that the left vertical map is injective, we deduce that  $g^*(w) = w_1^{\RP^2} \in H^1(\Z/2,\Z/2)$ and $g^*(v) = w_2^{\RP^2} \in H^2(\Z/2,\Z/2)$. Hence 
  \[g^*(w^3 - wv) = (w_1^{\RP^2})^3 - w_1^{\RP^2}w_2^{\RP^2} \in H^3(\Z/2;\Z/2).\] 
Since $w_1^{\RP^2} \neq 0$, it follows that $(w_1^{\RP^2})^3 \neq 0$, and we saw above that $w_2^{\RP^2}=0$.  
     It follows that $(w_1^{\RP^2})^3 \neq w_1^{\RP^2}w_2^{\RP^2}$, so $w^3 \neq wv \in H^3(G;\Z/2)$. 
    By \cref{thm:stably-exotic-dim3}, there are no closed stably exotic $4$-manifolds with normal 1-type data $(G,w,v)$. It now follows from \cref{thm:main-characterisation}~\eqref{item-main-characterisation-i} that $Y$ does not admit stably exotic fillings.
\end{proof}

In the notation of the previous proof, since $f^*\big(w_1(\nu_Y)^2 + w_2(\nu_Y)\big) = w_1(\nu_{\RP^2})^2 + w_2(\nu_{\RP^2}) = x^2 \neq 0$, it follows that $Y$ does not admit a normal $\pinminus$ structure, hence $Y$ does not admit a tangential $\pinplus$ structure by \cref{cor:equiv-pinplus-minus}. This aligns with \cref{thm:some-stably-exotic-fillings-exist}. 

\section{The cardinality of stably exotic fillings}\label{sec:number-fillings}

We consider the size of the set of stably exotic fillings of a given 3-manifold. First we show there are at most two fillings up to stable diffeomorphism, within a fixed stable homeomorphism class. Then we show that if $Y$ admits stably exotic fillings, then it admits infinitely many stable homeomorphism classes of such. 

\begin{proposition}\label{prop-at-most-two} 
Let $Y$ be a closed 3-manifold and let $M$ be a filling of $Y$. Let $\mathcal{F}_Y(M)$ denote the set of smooth fillings of $Y$ that are stably homeomorphic to $M$ rel.\ boundary,  modulo stable diffeomorphism rel.\ boundary. Then $|\mathcal{F}_Y(M)| \leq 2$.
\end{proposition}

\begin{proof}
After possibly stabilising,
we can suppose that $Y$ and $M$ are such that there exists $M'$ that is homeomorphic to~$M$ rel.\ boundary but not stably diffeomorphic rel.\ boundary; if not then  $|\mathcal{F}_Y(M)| =1$ and we are done.  
Let $\xi_G \colon B_G \to \BO$ be the normal 1-type of $M$, constructed using the normal 1-type data $(G:=\pi_1(M), w_1^M, w_2^M)$, and let $\ol{\nu}_M$ be a normal 1-smoothing. Similarly, let $\xi_G^\top \colon B^\top_G \to B\Top$ be the topological normal 1-type of $M$, constructed in the same way.

By \cref{thm:kreckstable1type}, the forgetful map $\Omega_4(\eta_G)\colon \Omega_4(\xi_G)\to\Omega_4(\xi^{\top}_G)$ is not injective and hence the kernel is isomorphic to $\Z/2$, generated by $K_3$ (with some normal~$1$-smoothing) by \cite{Kasprowski-Powell}*{Proposition~4.1}.

Now let $N$ be a smooth filling of $Y$ that is stably homeomorphic to $M$ rel.\ boundary. We will show that $N$ is stably diffeomorphic to either $M$ or $M'$ rel.\ boundary, which will complete the proof. By \cref{thm:kreckstable1type}\,\eqref{item:kreck-thm-ii}, there exists a topological normal $1$-smoothing $\ol{\nu}_N^{\top}$ of $N$ that is bordant to $\overline{\nu}^{\top}_M:=\eta_G\circ \ol{\nu}_M$ rel.\ boundary.
We claim that $\ol{\nu}_N^{\top}$ can be lifted to a normal $1$-smoothing of~$N$ which agrees with $\ol{\nu}_M$ on $Y$. To see this, recall from~\eqref{eqn:defn-of-B_G} that $B_G$ is the homotopy fibre of the map $p:=P_2(-)+(w_1^M\times w_2^M)$. Similarly, from~\eqref{eqn:defn-of-B_G-top}, we have that $B_G^{\top}$ is the homotopy fibre of the map $p^{\top}:=P_2(-)+(w_1^M\times w_2^M)$. Modelling these homotopy fibres as mapping path spaces, and using superscript $I$ to denote the free path space, we wish to first solve the lifting problem
\[
\begin{tikzcd}
    Y\ar[rrr, "(\overline{\nu}_M)|_Y"]\ar[d, "\mathrm{incl.}"]
    &&&
    (\BO\times K(G,1))\times_p(K(\Z/2,1)\times K(\Z/2,2))^I\ar[d, "\eta_G=(\eta\times\Id)\times \Id"]
    \\
     N\ar[rrr, "\overline{\nu}^{\top}_N"']\ar[urrr, dashed,"\ol{\nu}_N"]
     &&&
    (\BTop\times K(G,1))\times_{p^{\top}}(K(\Z/2,1)\times K(\Z/2,2))^I,
\end{tikzcd}
\]
where $\eta\colon \BO\to\BTop$ denotes the forgetful map.
In this model, $\overline{\nu}^{\top}_N=(\nu^{\top}_N\times f)\times\chi$, for some maps $f\colon N\to K(G,1)$ and $\chi\colon N\to (K(\Z/2,1)\times K(\Z/2,2))^I$ such that $p^{\top}((\nu_N^{\top}\times f)(x))=\chi(x)(1)$ for all $x\in N$. Noting that $(w_1\times w_2)\circ\eta\circ\nu_N=(w_1\times w_2)\circ\nu^{\top}_N$, we see that $\overline{\nu}_N:=(\nu_N\times f)\times\chi$ is a lift of $\overline{\nu}_N^{\top}$. 
To see that $\ol{\nu}_N$ agrees with $\ol{\nu}_M$ on restriction to $Y$, we write 
$\overline{\nu}_M=(\nu_M\times f_M)\times\chi_M$ for some $f_M$ and $\chi_M$, and compare to $(\nu_N\times f)\times\chi$. The maps $\nu_N$ and $\nu_M$ agree on restriction to $Y$ because $Y$ is the common boundary of $N$ and $M$. We also have $f|_Y=(f_M)|_Y$ and $\chi|_Y=(\chi_M)|_Y$, because $\ol{\nu}_M^{\top}=(\nu_M^{\top},f,\chi)$ and  $\ol{\nu}_N^{\top}=(\nu_N^{\top},f_N,\chi_N)$, and these topological normal $1$-smoothings agree on restriction to $Y$, by definition of~$\overline{\nu}^{\top}_N$. Thus $(\ol{\nu}_N)|_Y=(\ol{\nu}_M)|_Y$.
Finally, $\ol{\nu}_N$ is a normal $1$-smoothing because~$\eta_G$ induces an isomorphism on fundamental groups, and $\pi_2(B_G)=0$, so $\ol{\nu}_N$ inherits the required connectivity 
 properties from $\overline{\nu}_N^{\top}$.

By definition of $\ol{\nu}^{\top}_N$, we have that
$[M \cup_Y N, \ol{\nu}_M \cup - \ol{\nu}_N] \in \Omega_4(\xi_G)$ maps trivially to $\Omega_4(\xi^{\top}_G)$.
If $[M \cup_Y N, \ol{\nu}_M \cup - \ol{\nu}_N] = 0 \in \Omega_4(\xi_G)$ then $N$ is stably diffeomorphic to $M$ rel.\ boundary by \cref{thm:kreckstable1type}\,(i).  If not, $[M \cup_Y N, \ol{\nu}_M \cup - \ol{\nu}_N]$ is nonzero in $\ker \Omega_4(\eta_G) \cong \Z/2$. 
Since $M'$ is not stably diffeomorphic to $M$ rel.\ boundary,~\cref{thm:kreckstable1type}\,(i) implies that $[M' \cup_Y M,\ol{\nu}_{M'} \cup - \ol{\nu}_{M}]$ is nontrivial in $\Omega_4(\xi_G)$ for any normal 1-smoothing $\ol{\nu}_{M'}$ that agrees with $\ol{\nu}_M$ on $Y$. 
Since $M$ and~$M'$ are homeomorphic, by the argument of the previous paragraph, with $M'$ in place of $N$, there exists  such a normal 1-smoothing $\ol{\nu}_{M'}$ such that 
$[M' \cup_Y M,\ol{\nu}_{M'} \cup - \ol{\nu}_{M}]\in \ker \Omega_4(\eta_G)$.

Thus we have two representatives of the nontrivial element of $\ker \Omega_4(\eta_G) \cong \Z/2$, and so  the sum
\begin{equation}\label{eqn:element-of-bordism-group}
    [M' \cup_Y M,\ol{\nu}_{M'} \cup - \ol{\nu}_{M}] + 
[M \cup_Y N, \ol{\nu}_M \cup - \ol{\nu}_N] \in \Omega_4(\xi_G)
\end{equation}
is trivial. 
Take the product of $(M' \cup_Y M) \sqcup (M \cup_Y N)$ with $I$, and glue $(M \times I, \ol{\nu}_M \circ \pr_1)$ to \[\big((M' \cup_Y M) \cup (M \cup_Y N)\big) \times \{1\}\] along the two copies of $M$, to obtain a bordism over $\xi_G$, witnessing the fact that \eqref{eqn:element-of-bordism-group} is equal to $[M' \cup_Y N , \ol{\nu}_{M'} \cup - \ol{\nu}_N]$.  As the bordism class in~\eqref{eqn:element-of-bordism-group} is trivial, we have that $N$ is stably diffeomorphic rel.\ $Y$ to $M'$ by \cref{thm:kreckstable1type}. It follows that $\mathcal{F}_Y(M) = \{M,M'\}$, which completes the proof. 
\end{proof}

Now we show, as promised, that whenever we have a pair of stably exotic fillings for a 3-manifold~$Y$, we can find infinitely many stable homeomorphism classes of pairs of stably exotic fillings. 

\begin{proposition}\label{prop:many}
Let $M$ and $M'$ be stably exotic fillings of $Y$. Then for all $k \geq 0$, 
\[M_k := M \# k(S^1 \times S^3)\,\, \text{ and }\,\, M_k' := M' \# k(S^1 \times S^3)\]
are also stably exotic fillings of $Y$. Thus there are infinitely many stable homeomorphism classes of fillings of $Y$, each of which contains a stably exotic pair. 
\end{proposition} 

\begin{proof}
Since $M$ and $M'$ are stably homeomorphic rel.\ $Y$, so are $M_k$ and $M_k'$.  
   
    Let $H:= \pi_1(M)$ and let $\xi_H \colon B_H \to \BO$ be the normal 1-type of $M$ and $M'$, with normal 1-type data $(H,x,y)$. Let $F_k$ denote the free group on $k$ letters and let $G := H *F_k$. 
Note that $H^1(G;\Z/2) \cong H^1(H;\Z/2) \oplus (\Z/2)^k$ and $H^2(G;\Z/2) \cong H^2(H;\Z/2)$. 
Under these identifications let $w:= (x,0) \in H^1(G;\Z/2)$ and let $v := y \in H^2(G;\Z/2)$.  Since $\#^k (S^1 \times S^3)$ is spin, the normal $1$-type data of $M_k$ and $M'_{k}$ is $(G,w,v)$.
Let $\phi \colon G \to H$ be the projection. By \cref{lemma:induced-map-on-1-types}, there exists a map $\phi_B\colon B_G\to B_H$ covering $\Id_{\BO}\times \phi\colon \BO\times K(G,1)\to\BO\times K(H,1)$ up to homotopy. This map induces
\[(\phi_{B})_* \colon \Omega_4(\xi_G) \to \Omega_4(\xi_H).\]
To show $M_k$ and $M_k'$ are not stably diffeomorphic rel.\ $Y$, consider the union 
\[[M_k \cup_Y M_k', \ol{\nu}_{M_k} \cup - \ol{\nu}_{M_k'}] \in \Omega_4(\xi_{G})
\]
for arbitrary normal 1-smoothings $\ol{\nu}_{M_k}$ and $\ol{\nu}_{M_k'}$ that agree on $Y$. By \cref{thm:kreckstable1type}, we need to show that this is nontrivial. 
It suffices to show that its image
\begin{equation}\label{eqn:union-M-kM-k-prime}
[M_k \cup_Y M_k', \phi_B\circ\ol{\nu}_{M_k} \cup - \phi_B\circ\ol{\nu}_{M_k'}] \in \Omega_4(\xi_{H}).
\end{equation}
under $(\phi_{B})_*$ is nontrivial.

We claim that for each $S^1 \times S^3$ summand,  $S^1 \times S^3 = \partial(S^1 \times D^4)$ gives a $\xi_H$ null-bordism of $[S^1\times S^3,\phi_B\circ\ol{\nu}_{M_k}|_{S^1 \times S^3}]$.
To see this, consider the map of fibrations:
\[\begin{tikzcd}
    \BSpin \ar[r] \ar[d] & \BO \ar[r] \ar[d]  & K(\Z/2,1) \times K(\Z/2,2) \ar[d,"="] \\
    B_H\ar[r]& \BO \times K(H,1) \ar[r]  & K(\Z/2,1) \times K(\Z/2,2).
\end{tikzcd}\]
Since the $S^1\times S^3$ summand maps trivially to $K(H,1)$, and is spin, the restriction of $\phi_B\circ\ol{\nu}_{M_k}$ to $S^1 \times S^3$ lifts to $\BSpin$. 
Noting that $S^1 \times D^4$ is a  spin filling, for either choice of spin structure on $S^1 \times S^3$, it follows that $S^1 \times S^3$ is null-bordant over $B_H$ as claimed. 

It follows from the claim that $M_k$ is bordant rel.\ $Y$ to $M$ over $\xi_H$, for some normal 1-smoothing $\ol{\nu}_M \colon M \to B_H$. Similarly, $M'_k$ is bordant rel.\ $Y$ to $M'$ over $\xi_H$, for some normal 1-smoothing $\ol{\nu}_{M'} \colon M' \to B_H$. Moreover, by construction, \eqref{eqn:union-M-kM-k-prime} equals \[[M \cup_Y M',\ol{\nu}_M \cup - \ol{\nu}_{M'}] \in \Omega_4(\xi_H).\]   

Since $M$ and $M'$ are not stably diffeomorphic rel.\ $Y$, \cref{thm:kreckstable1type} implies that the latter union is nontrivial in $\Omega_4(\xi_H)$, as desired. 
\end{proof}

\def\MR#1{}
\bibliography{bib-stable-exotic}

\end{document}